\documentclass[a4paper,reqno,11pt]{amsart}

\usepackage{amsmath, amssymb, amsthm, enumerate,amsfonts,latexsym, mathrsfs}
\usepackage{stmaryrd}

\usepackage[top=30mm,right=30mm,bottom=30mm,left=30mm]{geometry}

\newtheorem{theorem}{Theorem}[section]

\newtheorem{lemma}[theorem]{Lemma}
\newtheorem{proposition}[theorem]{Proposition}

\newtheorem{Theorem}{Theorem}
\newtheorem{Corollary}[Theorem]{Corollary}
\newtheorem{Remark}[Theorem]{Remark}

\makeatletter
\newcommand{\imod}[1]{\allowbreak\mkern4mu({\operator@font mod}\,\,#1)}
\makeatother

\newcommand{\la}{\langle}
\newcommand{\ra}{\rangle}
\newcommand{\Z}{\mathbb{Z}}

\newcommand{\Irr}{{\mathrm {Irr}}}

\newcommand{\Centralizer}{{\mathrm {C}}}
\newcommand{\Normalizer}{{\mathrm {N}}}

\newcommand{\GL}{{\mathrm {GL}}}

\newcommand{\SL}{{\mathrm {SL}}}
\newcommand{\Sp}{{\mathrm {Sp}}}

\renewcommand{\a}{\alpha}
\renewcommand{\b}{\beta}
 \newcommand{\e}{\epsilon}
 
 \renewcommand{\O}{\Omega}
 
 \renewcommand{\to}{\rightarrow}

\newcommand{\leqs}{\leqslant}
\newcommand{\geqs}{\geqslant}

 \newcommand{\vs}{\vspace{2mm}}
\renewcommand{\la}{\langle}
\renewcommand{\ra}{\rangle}

\newcommand{\normeq}{\trianglelefteqslant}

\theoremstyle{definition}

\newtheorem{rem}[theorem]{Remark}

\begin{document}

\title[Derangements in primitive groups]{Derangements in primitive permutation groups, \\
with an application to character theory}

\author{Timothy C. Burness}
\email{t.burness@bristol.ac.uk }
\address{T.C. Burness, School of Mathematics, University of Bristol, Bristol BS8 1TW, UK}

\author{Hung P. Tong-Viet$^\dag$}
\email{ptongviet@math.uni-bielefeld.de}
\address{H.P. Tong-Viet, Fakult\"at f\"ur Mathematik, Universit\"at Bielefeld, D-33501 Bielefeld, Germany}

\thanks{$^{\dag}$ Supported by the project C$13$ `The geometry and combinatorics of groups' within the CRC 701.}


\date{\today}


\begin{abstract}
Let $G$ be a finite primitive permutation group and let $\kappa(G)$ be the number of conjugacy classes of derangements in $G$. By a classical theorem of Jordan, $\kappa(G) \geqs 1$. In this paper we classify the groups $G$ with $\kappa(G)=1$, and we use this to obtain new results on the structure of finite groups with an irreducible complex character that vanishes on a unique conjugacy class. We also obtain 
detailed structural information on the groups with $\kappa(G)=2$, including a complete classification for almost simple groups.
\end{abstract}

\maketitle


\section{Introduction}\label{s:intro}

Let $G$ be a transitive permutation group on a finite set $\Omega$ of size $n\geqs 2$, and let $H$ be the stabiliser of a point. An element $x\in G$ is a \emph{derangement} if it acts
fixed-point-freely on $\Omega$, or equivalently, if $x^G\cap H$ is empty, where $x^G$ is the conjugacy class of $x$ in $G$. The existence of derangements is guaranteed by a classical theorem of Jordan \cite{Jordan}, and we will write
$$\Delta(G) = G \setminus \bigcup_{g \in G}H^g$$
for the set of derangements in $G$. As discussed by Serre \cite{Serre}, Jordan's theorem has many
interesting applications in number theory and topology.

Various extensions and generalisations of Jordan's theorem have been studied in recent years. For
example, let $\delta(G) = |\Delta(G)|/|G|$ be the proportion of derangements in $G$. By a theorem
of Cameron and Cohen \cite{CC}, $\delta(G) \geqs 1/n$ and equality holds if and only if $G$ is
sharply $2$-transitive (that is, either $(G,n) = (S_2,2)$ or $G$ is a Frobenius group of order
$n(n-1)$ with $n$ a prime power). Using the Classification of Finite Simple Groups (CFSG),
Guralnick and Wan \cite{GW} have established the better bound $\delta(G) \geqs 2/n$ (with prescribed
exceptions), and a very recent theorem of Fulman and Guralnick (see \cite{FG1, FG2, FG3, FG4})
states that there is an absolute constant $\e>0$ such that $\delta(G)>\e$ for any simple transitive
group $G$. This latter result confirms a conjecture of Boston et al. \cite{Boston} and Shalev.

In a different direction, one can consider the existence of derangements of a given order. By a
theorem of Fein, Kantor and Schacher \cite{FKS}, $G$ contains a derangement of prime-power order
(their proof requires CFSG), and this result has important number-theoretic applications. However,
$G$ may not contain a derangement of prime order, and in this situation we say that $G$ is
\emph{elusive}. The first construction of elusive groups was presented in \cite{FKS}: let $p$ be a
Mersenne prime and take $G={\rm AGL}_{1}(p^2)$ and $H={\rm AGL}_{1}(p)$, so $n=p(p+1)$ and $G$ is
elusive since all elements of order $2$ or $p$ are conjugate in $G$. In \cite{Giudici}, Giudici
classifies the quasiprimitive elusive groups, and it follows that the $3$-transitive action of the
smallest Mathieu group ${\rm M}_{11}$ on $12$ points is the only almost simple primitive elusive
group. In \cite{IKLM}, the transitive groups $G$ in which all derangements are involutions are
determined; $G$ is either an elementary abelian $2$-group, or a Frobenius group with kernel an elementary abelian $2$-group.

In this paper, we are interested in the number of conjugacy classes of derangements in $G$, which we denote by $\kappa(G)$ (note that $\Delta(G)$ is a normal subset of $G$). By Jordan's theorem,
$\kappa(G) \geqs 1$. Our first result determines the primitive groups with $\kappa(G)=1$, and our
second gives a stronger result for almost simple groups.

\begin{Theorem}\label{t:main1}
Let $G$ be a finite primitive permutation group with point stabiliser $H$. Then $\kappa(G)=1$ if
and only if $G$ is sharply $2$-transitive, or $(G,H) = (A_5,D_{10})$ or $({\rm
L}_{2}(8){:}3,D_{18}{:}3)$.
\end{Theorem}

\begin{Theorem}\label{t:main2}
Let $G$ be a finite almost simple primitive permutation group with point stabiliser $H$. Then
either $\kappa(G) \geqs 3$, or $(\kappa(G),G,H)$ is recorded in Table \ref{t:kappa}. Moreover,
$\kappa(G)$ tends to infinity as $|G|$ tends to infinity.
\end{Theorem}

Note that in Table \ref{t:kappa}, we write $G=A_5$ rather than ${\rm L}_{2}(4)$ or ${\rm
L}_{2}(5)$. Similarly, we write $G=A_6$ rather than ${\rm L}_{2}(9)$ or ${\rm PSp}_{4}(2)'$, etc.
In addition, in the final row we write $[16]$ to denote a Sylow $2$-subgroup of ${\rm M}_{10} = A_6\cdot 2$.

\renewcommand{\arraystretch}{1.2}
\begin{table}
$$\begin{array}{cl} \hline
\kappa(G) & (G,H) \\ \hline
1 & (A_5,D_{10}),\; ({\rm L}_{2}(8){:}3, D_{18}{:}3) \\
2 & (A_5, D_6),\, (A_5, A_4),\, (S_5,D_{12}),\, (S_5,S_4),\, (A_6,3^2{:}4),\, (A_6,A_5),\, (S_6,3^2{:}D_8) \\
& ({\rm M}_{10},[16]),\, ({\rm L}_{2}(7), 7{:}3),\,({\rm L}_{2}(7),S_4),\, ({\rm L}_{3}(4),2^4{:}A_5),\, ({}^2B_2(8){:}3, 5{:}4 \times 3) \\ \hline
\end{array}$$
\caption{$\kappa(G)<3$, $G$ almost simple primitive}
\label{t:kappa}
\end{table}
\renewcommand{\arraystretch}{1}

By considering the cases in Table \ref{t:kappa}, we easily deduce the following corollary.

\begin{Corollary}\label{c:main_as}
Let $G$ be a finite almost simple transitive permutation group with point stabiliser $H$. Assume
$G$ is imprimitive. Then either $\kappa(G) \geqs 3$, or $\kappa(G) = 2$ and $(G,H) = (A_5,\Z_5)$.
\end{Corollary}

We also investigate the structure of a general primitive permutation group $G$ with $\kappa(G)=2$. A version of our main result is Theorem \ref{t:main3} below (see Section \ref{s:2classes} for more details). Hering's classification \cite{Hering, Hering2} of the $2$-transitive affine permutation groups is a key tool in the proof. 

\begin{Theorem}\label{t:main3}
Let $G$ be a finite primitive permutation group of degree $n$ with point stabiliser $H$. If 
$\kappa(G)=2$, then one of the following holds:
\begin{itemize}\addtolength{\itemsep}{0.2\baselineskip}
\item[{\rm (i)}] $(G,n) = (\Z_3,3)$;
\item[{\rm (ii)}] $G$ is one of the almost simple groups recorded in Table \ref{t:kappa};
\item[{\rm (iii)}] $G=HN$ is an affine group, where $N$ is an elementary abelian $p$-group of order $n=p^k$, and one of the following holds:

\vspace{1mm}

\begin{itemize}\addtolength{\itemsep}{0.2\baselineskip}
\item[{\rm (a)}] $G$ is a Frobenius group with kernel $N$, $p$ is odd and $|H|=(n-1)/2$;
\item[{\rm (b)}] $G$ is a non-Frobenius $2$-transitive group, and either $G$ is recorded in Table \ref{tab:b}, or $G$ is soluble, $H \leqs {\rm \Gamma L}_{1}(p^k)$, $k$ is even and $|H|=2(n-1)$.
\end{itemize}
\end{itemize}
Moreover, any group $G$ as in {\rm (i)}, {\rm (ii)}, {\rm (iii)(a)} or Table \ref{tab:b} has the property $\kappa(G)=2$.
\end{Theorem}

\renewcommand{\arraystretch}{1.2}
\begin{table}
$$\begin{array}{lll} \hline
n & G &   \\ \hline
2^2 & 2^2{:}S_3 \cong S_4 & \mathcal{P}(4,2) \\
5^2 & 5^2{:}(2^{1+2}.6) & \mathcal{P}(5^2,17) \\
11^2 & 11^2{:}(2^{1+2}.[30]) & \mathcal{P}(11^2,42) \\
3^4 & 3^4{:}((2 \times Q_8){:}2){:}5 & \mathcal{P}(3^4,70) \\
29^2 & 29^2{:}(7 \times 2.{\rm SL}_{2}(5)) & \mathcal{P}(29^2,104) \\ \hline
\end{array}$$
\caption{Some affine $2$-transitive groups $G$ with $\kappa(G)=2$}
\label{tab:b}
\end{table}
\renewcommand{\arraystretch}{1}

\begin{Remark}
\emph{Let us make a couple of remarks on the statement of Theorem \ref{t:main3}.
\begin{itemize}\addtolength{\itemsep}{0.2\baselineskip}
\item[{\rm (a)}] In Table \ref{tab:b} we use the notation $\mathcal{P}(n,i)$ to denote the $i$-th primitive permutation group of degree $n$ in the 
library of primitive groups stored in {\sc Magma} \cite{magma}, which can be accessed via the command \textsf{PrimitiveGroup}$(n,i)$. 
\item[(b)] Consider part (iii)(b) of Theorem \ref{t:main3}, where $H \leqs {\rm \Gamma L}_{1}(p^k)$, $k$ is even and $|H|=2(p^k-1)$. Here it is difficult to give a complete description of the possibilities for $G$ with the property $\kappa(G)=2$, but we can show that $\kappa(G)=2$ in the special case $H = {\rm GL}_{1}(p^k)\cdot 2$ (see Proposition \ref{p:new}).
\end{itemize}}
\end{Remark}


One of our main motivations stems from an application to the character theory of finite groups. Let
$G$ be a finite nonabelian group and let $\chi\in\Irr(G)$ be a nonlinear irreducible complex character of $G$.
A classical theorem of Burnside \cite[Theorem 3.15]{Isaacs} states that $\chi(x)=0$ for some $x\in G$. In this situation, we say that $\chi$ \emph{vanishes} at $x$, and $x$ is called a \emph{zero} of
$\chi$. Since $\chi$ is a class function, it vanishes on the conjugacy class $x^G$, and we write
$n(\chi)$ for the number of conjugacy classes of $G$ on which $\chi$ vanishes. Therefore, Burnside's theorem states that $n(\chi)\geqs 1$ for all nonlinear $\chi\in\Irr(G)$. In fact, by a theorem of Malle, Navarro and Olsson \cite{MNO}, $\chi$ vanishes on some element of prime power order, and it is interesting to note that their proof uses the aforementioned theorem of Fein, Kantor and Schacher \cite{FKS} on derangements in transitive permutation groups.

Several authors have investigated the structure of finite groups with a nonlinear irreducible
character $\chi$ such that $n(\chi)$ is small, and there has been particular interest in the special case
$n(\chi)=1$. For example, Zhmud' \cite{Zhmud} obtained partial results on the structure of soluble groups with this property. In later work, Chillag \cite[Corollary 2.4]{Chillag} proved that if $G \neq G'$ then either $G$ is a Frobenius group with an abelian odd-order kernel of index two, or $\chi$ is
irreducible upon restriction to $G'$. In fact, if $G$ is any finite nonabelian group such that $n(\chi) \leqs 1$ for all $\chi \in \Irr(G)$, then $G$ is a Frobenius group with an abelian odd-order kernel of index two (see \cite[Proposition 2.7]{Chillag}; the proof uses CFSG). See \cite{Dixon} and \cite{Qian} for additional structural results on soluble groups with this extremal property.

Let us consider the general case: $G$ is a finite nonabelian group with a nonlinear irreducible
character $\chi$ such that $n(\chi)=1$. 
Recall that $\chi\in\Irr(G)$ is \emph{imprimitive} if it can be
induced from a character of a proper subgroup of $G$, i.e., $\chi=\phi^G$ for some $\phi\in\Irr(H)$
and proper subgroup $H$ of $G$. Otherwise, $\chi$ is \emph{primitive}.

Suppose $\chi\in\Irr(G)$ is a nonlinear imprimitive irreducible character such that $n(\chi)=1$, say $\chi=\phi^G$ where $\phi\in\Irr(H)$ and $H$ is a proper
subgroup of $G$. Set
$$\Delta_H(G):=G\setminus \bigcup_{g\in G} H^g.$$
Clearly, by definition of the induced character $\phi^G$, if $x\in \Delta_H(G)$
then $\chi(x)=0$ and thus $\Delta_H(G)=x^G$. Note that the converse does not hold in general; the condition $\Delta_H(G)=x^G$ does not imply that there is a character $\phi \in \Irr(H)$ such that $\phi^G \in \Irr(G)$ and $n(\phi^G)=1$. For example, if $(G,H)=(A_5,D_{10})$ then
$\Delta_H(G)=x^G$ by Theorem \ref{t:main1}, but no character of $H$ can be irreducibly induced to $G$ since $|G:H|=6$ and $\chi(1) \leqs 5$ for all $\chi \in \Irr(G)$.

If we assume further that $H$ is core-free and maximal, then $G$ is a primitive permutation group on $\Omega=G/H$ with $\kappa(G)=1$, so in this situation the possibilities
for $G$ and $H$ are given by Theorem \ref{t:main1}. 

In general, the structure of $G$ can be more complicated. In Theorem \ref{t:main_app} below we  describe the normal structure of finite groups $G$ with the property that $n(\chi) = 1$ for some nonlinear imprimitive irreducible character $\chi = \phi^G$, where $\phi \in {\rm Irr}(H)$ 
for some maximal subgroup $H$ of $G$. 
In the statement of Theorem \ref{t:main_app}, recall that a finite group $G$ is a \emph{Camina group} if $|\Centralizer_G(x)| = |\Centralizer_{G/G'}(G'x)|$ for all $x \in G \setminus G'$.

\begin{Theorem}\label{t:main_app}
Let $H$ be a maximal subgroup of a finite group $G$ such that $n(\chi) = 1$ for a nonlinear imprimitive irreducible character $\chi = \phi^G$ with $\phi \in {\rm Irr}(H)$. Write $\Delta_H(G)=x^G$ and let $N=H_G$ denote the normal core of $H$. Then one of the following holds:
\begin{itemize}\addtolength{\itemsep}{0.2\baselineskip}
\item[{\rm (i)}] $G$ is a Frobenius group with an abelian odd-order kernel $H=G'$ of index two.
\item[{\rm (ii)}] $G/N$ is a $2$-transitive Frobenius group with an elementary abelian kernel $M/N$ of order $p^n$ for some prime $p$ and integer $n\geqs 1$, and a complement $H/N$ of
order $p^n-1$. Moreover, $x^G=M \setminus N$, $|\Centralizer_G(x)|=p^n$,  $|x^G|=|H|$, $M'=N$ and one of the following holds:

\vspace{1mm}

\begin{itemize}\addtolength{\itemsep}{0.2\baselineskip}
\item[{\rm (a)}] $M$ is a Frobenius group with kernel $M'$ and $p^n=p>2$.
\item[{\rm (b)}] $M$ is a Frobenius group with kernel $K\normeq G$ such that $G/K\cong \SL_2(3)$ and $M/K\cong Q_8$.
\item[{\rm (c)}] $M$ is a Camina $p$-group.
\end{itemize}

\item[{\rm (iii)}] $G/N\cong {\rm L}_2(8){:}3$, $H/N\cong D_{18}{:}3$, $N$ is a nilpotent $7'$-group and $\Centralizer_G(x)=\la x\ra \cong \Z_7$.

\item[{\rm (iv)}] $G/N\cong A_5$, $H/N\cong D_{10}$, $N$ is a $2$-group and $\Centralizer_G(x)=\la x\ra \cong \Z_3$.
\end{itemize}
In particular, if $G=G'$ then either case {\rm (ii)(c)} holds with $p^n=11^2$ and $G/N\cong
11^2{:}\SL_2(5)$, or case {\rm (iv)} holds.
\end{Theorem}


\begin{Remark}\label{rem:main}
\emph{Let us make some remarks on the statement of Theorem \ref{t:main_app}.
\begin{itemize}\addtolength{\itemsep}{0.2\baselineskip}
\item[{\rm (a)}] Firstly, observe that there is no loss in assuming that $H$ is a maximal subgroup of $G$. Indeed, if $n(\chi)=1$ and 
$\chi=\lambda^G$ for some $\lambda\in\Irr(J)$ and proper subgroup $J<G$, then 
$\chi=(\lambda^H)^G=\phi^G$, whenever $J \leqs H< G$ with $\phi=\lambda^H\in\Irr(H)$.  
\item[{\rm(b)}] For imprimitive characters, Theorem \ref{t:main_app} extends several known results in the literature. For example, the conclusion in part (i) coincides with the first part of \cite[Corollary 2.4]{Chillag}, and parts (i) and (ii)(a,b) are exactly the conclusions (1)-(3) in \cite[Theorem 1.1]{Qian} (see also \cite[Theorem 9]{Dixon}). It is worth noting that the relevant  results in \cite{Dixon, Qian} only apply in the case $G$ is soluble, whereas Theorem \ref{t:main_app} holds for any finite group $G$.
\item[{\rm(c)}] In Section \ref{s:zeros} we prove Theorem \ref{t:main_app} under a weaker assumption, namely, we only require that $G$ is a finite subgroup with a maximal subgroup $H$ such that $\Delta_H(G)=x^G$ for some $x \in G$. 
\item[{\rm (d)}] In parts (iii) and (iv), we note that the core $N=H_G$ is nontrivial since the index $|G:H|$ is larger than any character degree of $G/N$.
\item[{\rm(e)}] This structure theorem is an important step towards a complete classification of the finite groups with a nonlinear irreducible character that vanishes on a unique conjugacy class. Indeed, in a forthcoming paper, we study the structure of the groups arising in parts (ii)(c), (iii) and (iv) in more detail, and we will also consider the primitive case in future work.
\end{itemize}}
\end{Remark}

\vs

Finally, let us make some comments on the notation and organisation of the paper. Our group-theoretic notation is fairly standard. In particular, we use the notation of Kleidman and Liebeck \cite{KL} for simple groups and their automorphism groups; for example, we write ${\rm L}_{n}(q)$ and ${\rm U}_{n}(q)$ for ${\rm PSL}_{n}(q)$ and ${\rm PSU}_{n}(q)$, respectively. We use $\Z_n$, or just $n$, to denote a cyclic group of order $n$, and $(a,b)$ denotes the highest common factor of the positive integers $a$ and $b$.

In Section \ref{s:red} we establish a useful result that immediately reduces the proof of Theorem \ref{t:main1} to almost simple groups. We focus on the almost simple groups in Section \ref{s:as}, where we complete the proofs of Theorem \ref{t:main1} and \ref{t:main2}. The structure of the primitive groups $G$ with $\kappa(G)=2$ is investigated in Section \ref{s:2classes}, and we establish Theorem \ref{t:main3}. Finally, in Section \ref{s:zeros} we prove Theorem \ref{t:main_app} on the finite groups $G$ with a maximal subgroup $H$ and a nonlinear imprimitive irreducible character $\chi = \phi^G$ such that $n(\chi)=1$ and $\phi \in {\rm Irr}(H)$.

\section{A reduction theorem}\label{s:red}
Let $G \leqs {\rm Sym}(\Omega)$ be a transitive permutation group of degree $n$ with point
stabiliser $H = G_{\a}$. Recall that $G$ is a \emph{Frobenius group} if $G$ is not regular and only
the identity element has more than one fixed point (equivalently, $H \neq 1$ and $H \cap H^g = 1$
for all $g \in G \setminus H$). In this situation, $N := \{1\}\cup \Delta(G)$ is a regular normal
subgroup of $G$ (see \cite[Theorem 7.2]{Isaacs}, for example) and we have $G=HN$ and $H \cap N = 1$
(we call $N$ the \emph{Frobenius kernel} of $G$). Since $H$ acts semiregularly on $\O \setminus
\{\a\}$ it follows that $|G|=n(n-1)/d$, where $d$ divides $n-1$ ($d$ is the number of $H$-orbits on $\O \setminus \{\a\}$). If $G$ is $2$-transitive, i.e., if $H$
acts transitively on $\O \setminus \{\a\}$, then $d=1$ and it follows that any two nontrivial
elements of $N$ are conjugate in $G$ (so $N$ is an elementary abelian $p$-group for some prime $p$,
and $n$ is a power of $p$). In particular, if $G$ is a $2$-transitive Frobenius group then
$\kappa(G)=1$.

Also recall that $G$ is \emph{sharply $2$-transitive} if $G$ acts regularly on the set of pairs of
distinct elements of $\O$ (so $G$ is $2$-transitive and no nontrivial element of $G$ fixes more
than one point). In particular, $G$ is sharply $2$-transitive if and only if $(G,n) = (S_2,2)$ or
$G$ is a $2$-transitive Frobenius group. As noted above, the latter groups are precisely the
Frobenius groups of order $n(n-1)$ with $n$ a prime power.

Given a group $X$, we write $X^* = X \setminus \{1\}$ for the set of nontrivial elements in $X$.

\begin{theorem}\label{t:red}
Let $G \leqs {\rm Sym}(\Omega)$ be a finite primitive permutation group and assume $G$ is not
almost simple. Then $\kappa(G)=1$ if and only if $G$ is sharply $2$-transitive.
\end{theorem}

\begin{proof}
Let $H = G_{\a}$ be a point stabiliser of $G$ and let $n=|G:H|$ denote the degree of $G$. Suppose
$G$ is sharply $2$-transitive. The case $(G,n) = (S_2,2)$ is clear so let us assume $G$ is a
$2$-transitive Frobenius group with kernel $N$. Here $\Delta(G) = N^*$ and $H$ acts regularly on
$\Omega \setminus \{\a\}$, so $|H|=n-1$. Let $x \in N^*$. Then $\Centralizer_G(x) \leqs N$ and thus
$|x^G| \geqs |G:N|=|H| = |N^*|$. Since $N$ is normal we have $x^G \subseteq N^*$, so
$\Delta(G)=N^*=x^G$ and $\kappa(G)=1$.

Conversely, suppose $\kappa(G)=1$. Let $N$ be a minimal normal subgroup of $G$ and note that $N$ is
transitive and $G=HN$. There are two cases to consider.

First assume $N$ is regular, so $H\cap N=1$ and $N^* \subseteq \Delta(G)$. In fact, since
$\kappa(G)=1$, we have $N^* = x^G = \Delta(G)$ for some $x \in N^*$. If $N$ is nonabelian, then it
is isomorphic to a direct product of isomorphic nonabelian simple groups and hence $|N|$ is
divisible by at least three distinct primes, which is a contradiction since $N^*=x^G$. Therefore
$N$ is abelian and so $N\cong \Z_p^k$ for some prime $p$ and integer $k\geqs 1$. In particular, $N
\leqs \Centralizer_G(x)$. Now $|\Delta(G)| \geqs |H|$ by \cite{CC}, with equality if and only if
$G$ is sharply $2$-transitive. Therefore, $|x^G|=|G:\Centralizer_G(x)|\geqs |H|$ and thus $|N|\geqs
|\Centralizer_G(x)|$, so $\Centralizer_G(x)=N$ and $G$ is sharply $2$-transitive.

Now assume $H\cap N \neq 1$. It follows that $N\cong S^k$, where $S$ is a nonabelian simple group
and $k\geqs 1$. By \cite[Corollary 4.3B]{DM}, $N$ is the unique minimal normal subgroup of $G$. If
$k=1$ then $G$ is almost simple as $\Centralizer_G(N)=1$. So assume that $k\geqs 2$. Let $\pi_i$
denote the projection map from $H \cap N$ to the $i$-th simple factor of $N$. As noted in the proof
of \cite[Theorem 2.1]{BGW}, there exists a nontrivial subgroup $R$ of $S$ such that $\pi_i(H\cap N)
\cong R$ for all $1 \leqs i \leqs k$.

If $R=S$, then there exists a partition $\mathcal{P}$ of $\{1,2,\ldots, k\}$ such that $H\cap
N=\prod_{P\in \mathcal{P}}D_P$, where $D_P\cong S$ and $\pi_i(D_P)=S$ if $i\in P$, otherwise
$\pi_i(D_P)=1$. For each $P\in \mathcal{P}$, let $N_P$ be a subgroup given by the direct product of
$|P|-1$ of the simple direct factors of $N$ corresponding to $P$. Then
$N_0:=\prod_{P\in\mathcal{P}}N_P\leqs N$ has trivial intersection with $H$ and has order
$|\Omega|$. In particular, $N_0$ is a regular subgroup whose order is divisible by $|S|$. Since
$|S|$ is divisible by at least three distinct primes, it follows that $N_0$ has at least three
elements of distinct prime orders and thus $\kappa(G) \geqs 3$, a contradiction.

Finally, suppose $R\neq S$. By \cite[Theorem 4.6A]{DM}, $G\leqs L\wr S_k$ acting with its product
action on $\Omega=\Gamma^k$ for $k\geqs 2$, where $L\leqs \textrm{Sym}(\Gamma)$ is a primitive
almost simple group with socle $S$. If $u \in S$ is a derangement on $\Gamma$ then
$(u,1,1,\ldots,1), (u,u,1,\ldots,1) \in N$ are non-conjugate derangements on
$\Omega$, so $\kappa(G) \geqs 2$. This final contradiction completes the proof of the theorem.
\end{proof}

\section{Almost simple groups}\label{s:as}

In this section we will prove Theorem \ref{t:main2}. Let $G$ be a finite almost simple primitive
permutation group with socle $S$. Let $A={\rm Aut}(S)$, so $S \leqs G \leqs A$. In order to establish the bound $\kappa(G)
\geqs 3$ it suffices to show that if $H$ is a maximal subgroup of $S$ then there are at least three
$A$-classes of elements $x \in S$ such that $x^A \cap H$ is empty. Similarly, to justify the
asymptotic statement in Theorem \ref{t:main2},  we will show that the number of such $A$-classes tends to infinity as $|S|$
tends to infinity.

In view of Theorem \ref{t:red}, we see that Theorem \ref{t:main1} follows immediately from Theorem
\ref{t:main2}. Similarly, Corollary \ref{c:main_as} is easily deduced from Theorem \ref{t:main2}.

\subsection{Preliminaries}\label{ss:prel}

Here we record some preliminary results that will be useful in the proof of Theorem
\ref{t:main2}. Let $\phi:\mathbb{N} \to \mathbb{N}$ be Euler's totient function defined by
$$\phi(n) = |\{m \in \{1, \ldots, n-1\} \mid (m,n)=1\}|.$$
We will need the following elementary lower bound.

\begin{lemma}\label{l:euler}
If $n \in \mathbb{N}$ then $\phi(n) \geqs \sqrt{n/a}$, where $a=2$ if $n \equiv 2 \imod{4}$, otherwise $a=1$.
\end{lemma}

\begin{proof}

Write $n =\prod_{i}p_i^{a_i}$, where the $p_i$ are distinct primes, so
$$\phi(n) =  \prod_{i}\phi(p_i^{a_i}) = \prod_{i}p_i^{a_i-1}(p_i-1).$$
If $n \not \equiv 2 \imod{4}$ then $(p_i,a_i) \neq (2,1)$, so $p_i^{a_i-1}(p_i-1) \geqs
p_i^{a_i/2}$ and thus $\phi(n) \geqs \sqrt{n}$. Similarly, if $n \equiv 2 \imod{4}$ then $n=2m$ and
$m$ is odd, so $\phi(n) = \phi(m) \geqs \sqrt{m} = \sqrt{n/2}$.
\end{proof}

\begin{lemma}\label{l:class}
Let $x \in S$ be a self-centralising element of order $\a$ with $|\Normalizer_S(\la x \ra): \la x \ra|=n$.
Then there are at least $\phi(\a)/n|{\rm Out}(S)|$ distinct $A$-classes of such elements in $S$.
\end{lemma}

\begin{proof}
There are precisely $\phi(\a)$ elements in $\la x \ra$ of order $\a$, and for any such element $y$
we note that $|y^S \cap \la x \ra|=n$ since $\Centralizer_S(x) = \la x \ra$ and $|\Normalizer_S(\la x \ra): \la x
\ra|=n$. Therefore, $\la x \ra$ contains $\phi(\a)/n$ distinct $S$-class representatives of order
$\a$, so there are at least $\phi(\a)/n|{\rm Out}(S)|$ distinct $A$-classes.
\end{proof}

If $S$ is a simple group of Lie type then $|{\rm Out}(S)|$ is conveniently recorded in \cite[Tables 5.1.A, 5.1.B]{KL}.

Finally, let us introduce some additional notation. Let $G$ be an almost simple group with socle
$S$ and let $\mathcal{M}(G)$ be the set of maximal subgroups $H$ of $G$ such that $G=SH$. Given $H
\in \mathcal{M}(G)$, let $\kappa(G,H)$ denote the number of conjugacy classes of derangements in
$G$, with respect to the primitive action of $G$ on $G/H$. We define
\begin{equation}\label{e:phi}
\Phi(G) = \min\{\kappa(G,H) \mid H \in \mathcal{M}(G)\}.
\end{equation}
In addition, if $X$ is a finite group then $\pi(X)$ denotes the set of prime divisors of $|X|$.

\subsection{Sporadic groups}\label{ss:spor}

Here we establish Theorem \ref{t:main2} for sporadic groups. Set
$$\mathcal{A} = \{{\rm HS}.2, {\rm He}.2, {\rm Fi}_{22}.2,  {\rm HN}.2,  {\rm O'N}.2, {\rm Fi}_{24}, \mathbb{B}, \mathbb{M}\}.$$

\begin{proposition}\label{p:spor1}
The conclusion to Theorem \ref{t:main2} holds if $S$ is a sporadic simple group and $G \not\in \mathcal{A}$.
\end{proposition}

\begin{proof}
In each case it is straightforward to calculate the exact value of $\kappa(G,H)$ using the
information on the fusion of $H$-classes in $G$ that is available in the \textsf{GAPCTL} Character
Table Library \cite{GAPCTL}. For example, we obtain the following results when $G = {\rm M}_{11}$:
\small
\renewcommand{\arraystretch}{1.2}
$$\begin{array}{lccccc} \hline
H & {\rm M}_{10} & {\rm L}_{2}(11) & {\rm M}_{9}.2 & S_5 & 2.S_4 \\ \hline
\kappa(G,H) & 3 & 3 & 3 & 4 & 3 \\ \hline
\end{array}$$
\normalsize
In all cases, the exact value of $\Phi(G)$ (see \eqref{e:phi}) is recorded in Table \ref{t:phig}.
\end{proof}

\begin{table}
\footnotesize
\arraycolsep=4pt
\renewcommand{\arraystretch}{1.2}
$$\begin{array}{lccccccccccccccc} \hline
G & {\rm M}_{11} & {\rm M}_{22} & {\rm M}_{12} & {\rm J}_{1} & {\rm HS} & {\rm M}_{22}.2 & {\rm M}_{23} &
{\rm J}_{2} & {\rm M}_{12}.2 & {\rm M}_{24} & {\rm J}_{3} & {\rm McL} & {\rm McL}.2 & {\rm J}_{2}.2 & {\rm O'N} \\ \hline

\Phi(G) &  3 & 4 & 5 & 5 & 5 & 6 & 6 & 6 & 7 & 7 & 7 & 7 & 7 & 8 & 9 \\ \hline
& & & & & & & & & & & & &  & & \\ \hline

\multicolumn{1}{c}{{\rm Co}_{3}} & {\rm J}_{3}.2 & {\rm Th} & {\rm Co}_{2} & {\rm He} & {\rm Ru} & {\rm Ly} &
{\rm Fi}_{22} & {\rm Fi}_{23} & {\rm Suz} & {\rm J}_{4} & {\rm HN} & {\rm Suz}.2 & {\rm Fi}_{24}' & {\rm Co}_{1} \\ \hline

\multicolumn{1}{c}{10} & 11 & 11 & 12 & 13 & 13 & 14 & 14 & 15 & 15 & 17 & 19 & 21 & 26 & 27 \\ \hline
\end{array}$$
\caption{$\Phi(G)$ for some almost simple sporadic groups}
\label{t:phig}
\end{table}
\renewcommand{\arraystretch}{1}

\begin{proposition}\label{p:spor2}
The conclusion to Theorem \ref{t:main2} holds if $S$ is a sporadic simple group.
\end{proposition}

\begin{proof}

We may assume $G \in \mathcal{A}$. If $G \in \{{\rm HS}.2, {\rm He}.2, {\rm Fi}_{22}.2\}$ we can
use {\sc Magma} \cite{magma} to determine the fusion of $H$-classes in $G$, working with the
respective permutation representations of degree $100$, $2058$ and $3510$ provided in the Web-Atlas
\cite{WebAt}. In this way, we calculate that $\Phi({\rm HS}.2) = 11$, $\Phi({\rm He}.2) = 16$ and
$\Phi({\rm Fi}_{22}.2) = 17$.

Of course, we can immediately discard any remaining cases $(G,H)$ with the property that $|\pi(G)
\setminus \pi(H)| \geqs 3$, which eliminates the Baby Monster and the Monster. In fact, one can
check that it only remains to deal with the following cases:
$$\begin{array}{llllll}
(1) & ({\rm HN}.2,S_{12}) & (2) & ({\rm HN}.2, 4.{\rm HS}.2) & (3) & ({\rm HN}.2,{\rm U}_{3}(8){:}6) \\
(4) & ({\rm HN}.2, (5{:}4 \times {\rm U}_{3}(5)).2) & (5) & ({\rm O'N}.2,{\rm J}_{1}\times 2) & (6) & ({\rm Fi}_{24}, {\rm Fi}_{23} \times 2)
\end{array}$$
In cases (1) -- (3), the fusion of $H$-classes in $G$ is stored in \cite{GAPCTL} and the result
quickly follows as above (we get $\kappa(G,H) = 31, 23, 57$, respectively). In (4) and (6),
\cite[Proposition 4.3]{BGW} implies that $G$ contains at least three classes of  derangements of
prime order. For example, in (6) we find that $G$ contains derangements of order $3,7$ and $29$.
Similarly, in case (5), $G$ contains derangements of order $7$ and $31$, and elements of order $14$
are also derangements since $|H|$ is indivisible by $14$.
\end{proof}

\subsection{Alternating groups}\label{ss:alt}

In this section we establish Theorem \ref{t:main2} in the case where $S=A_n$ is an alternating group of degree $n \geqs 5$.

\begin{proposition}\label{p:an1}
The conclusion to Theorem \ref{t:main2} holds if $S=A_n$ and $n \leqs 24$.
\end{proposition}

\begin{proof}
We can use {\sc Magma} \cite{magma} to
determine the fusion of $H$-classes in $G$, and the result quickly follows. For instance, we obtain
the results presented in Table \ref{tab:an1} if $G \in \{A_5,S_5,A_6,S_6\}$. In addition, we
calculate that $\Phi(G)=4$ if $G={\rm PGL}_{2}(9) = A_6.2$ or ${\rm Aut}(A_6) = A_6.2^2$, and if $G
= {\rm M}_{10} = A_6.2$ we get $\kappa(G,[16]) = 2$ (where $[16]$ is a Sylow $2$-subgroup of $G$),
$\kappa(G,3^2{:}Q_8) = 3$ and $\kappa(G,5{:}4) = 4$. For $7 \leqs n \leqs 24$ we record $\Phi(G)$
in Table \ref{tab:an2}.
\renewcommand{\arraystretch}{1.2}
\begin{table}
\footnotesize
$$\begin{array}{cl} \hline
\kappa(G,H) & (G,H) \\ \hline
1 & (A_5,D_{10}) \\
2 & (A_5, D_6),\, (A_5, A_4),\, (S_5,D_{12}),\, (S_5,S_4),\, (A_6,3^2{:}4),\, (A_6,A_5),\, (S_6,3^2{:}D_8) \\
3 & (S_5,5{:}4),\, (A_6,S_4),\, (S_6, S_4 \times 2) \\
4 & (S_6,S_5) \\ \hline
\end{array}$$
\caption{$\kappa(G,H)$ for $G \in \{A_5,S_5,A_6,S_6\}$}
\label{tab:an1}
\end{table}
\renewcommand{\arraystretch}{1}
\end{proof}

\renewcommand{\arraystretch}{1.2}
\begin{table}
\footnotesize
$$\begin{array}{lcccccccccccccccccc} \hline
n & 7 & 8 & 9 & 10 & 11 & 12 & 13 & 14 & 15 & 16 & 17 & 18 & 19 & 20 & 21 & 22 & 23 & 24 \\ \hline
\Phi(A_n) & 3 & 5 & 5 & 7 & 7 & 11 & 12 & 15 & 18 & 22 & 26 & 31 & 38 & 46 & 55 & 62 & 74 & 88 \\
\Phi(S_n) & 4 & 5 & 7 & 9 & 11 & 15 & 19 & 23 & 30 & 35 & 44 & 50 & 65 & 80 & 95 & 111 & 133 & 157 \\ \hline
\end{array}$$
\caption{$\Phi(G)$ for $G \in \{A_n,S_n\}$, $7 \leqs n \leqs 24$}
\label{tab:an2}
\end{table}
\renewcommand{\arraystretch}{1}

\begin{proposition}\label{p:an2}
The conclusion to Theorem \ref{t:main2} holds if $S=A_n$.
\end{proposition}

\begin{proof}
We may assume that $n>24$.
Let $H$ be a maximal subgroup of $G$ such that $G=SH$. We consider three cases according to the action of $H$ on $\{1, \ldots, n\}$:
\begin{itemize}\addtolength{\itemsep}{0.3\baselineskip}
\item[(a)] $H$ acts primitively on $\{1, \ldots, n\}$;
\item[(b)] $H$ acts transitively and imprimitively on $\{1, \ldots, n\}$;
\item[(c)] $H$ acts intransitively on $\{1, \ldots, n\}$.
\end{itemize}

In case (a), let $x \in G$ be an $r$-cycle, where $r$ is a prime such that $2 \leqs r<n-2$. Then  a
theorem of Jordan \cite{jordan} implies that $x$ is a derangement, whence $\kappa(G,H) \geqs 3$.

Next assume (b) holds, so $H$ is of type $S_a \wr S_b$, where $n=ab$ and $a,b \geqs 2$. Let $r$ be
a prime in the interval $(a,n)$. As noted in the proof of \cite[Proposition 3.6]{BGW}, any
$r$-cycle in $G$ is a derangement. Now, if $a \geqs 9$ then there are at least three distinct
primes in the interval $(a,2a)$ (see \cite{Ram}, for example) and the result follows (we also note
that the number of primes in $(a,2a)$ tends to infinity as $a$ tends to infinity). Similarly, if
$a<9$ then $b \geqs 4$ (since $n>24$) and there are at least three primes in $(a,4a)$ for all $a
\geqs 2$.

Finally, let us consider (c), so $H$ is of type $S_k \times S_{n-k}$ with $1 \leqs k < n/2$.
Clearly, if $n$ is even then any $x \in G$ of cycle-shape $(\ell,n-\ell)$, where $1 \leqs \ell
\leqs n/2$, $\ell \neq k$, is a derangement. Now assume $n$ is odd. If $k \neq 3$ then any $x \in
G$ of cycle-shape $(3,\ell,n-\ell-3)$, where $1 \leqs \ell \leqs (n-3)/2$, $\ell \not\in
\{k,k-3\}$, is a derangement. Similarly, if $k=3$ then take $x \in G$ of cycle-shape
$(5,\ell,n-\ell-5)$, where $1 \leqs \ell \leqs (n-5)/2$ and $\ell \neq 3$. The result follows.

In each case, note that we have also shown that $\kappa(G,H)$ tends to infinity as $|G|$ tends to infinity.
\end{proof}

\vs

For the remainder, we may assume that $S$ is a group of Lie type; we deal with the exceptional
groups in Section \ref{ss:excep} and the classical groups in Section \ref{ss:class}. Our basic
approach is similar in both cases. The aim is to identify a collection of elements in $G$ that
belong to very few maximal subgroups -- if we can show that there are at least three $A$-classes of
such elements (and the number of such classes tends to infinity as $|G|$ tends to infinity), then
it just remains to deal with the specific possibilities for $H$ that contain these elements. Given
such a subgroup $H$, we choose an alternative collection of elements $x \in G$ such that $x^A \cap
H$ is empty, and we then show that there are sufficiently many $A$-classes with this
property. For some groups of low rank over small fields, we will use {\sc Magma} \cite{magma} to
obtain the desired result.

\subsection{Exceptional groups}\label{ss:excep}

Let $S$ be a finite simple group of exceptional Lie type over $\mathbb{F}_q$, where $q=p^f$ and $p$ is a prime. Set
$$\mathcal{A} = \{G_2(3), G_2(4), G_2(5), {}^2B_2(8), {}^2B_2(32), {}^2G_2(27), {}^3D_4(2),  {}^2F_4(2)'\}.$$

\begin{proposition}\label{p:ex}
If $S \in \mathcal{A}$ then either $\kappa(G,H) \geqs 4$, or $(G,H) = ({}^2B_2(8){:}3, 5{:}4 \times 3)$ and $\kappa(G,H) = 2$.
\end{proposition}

\begin{proof}
This is a straightforward calculation, using {\sc Magma} and a suitable permutation representation of $G$ given in the Web-Atlas \cite{WebAt}.
\end{proof}

For the remainder, we may assume that $S \not\in \mathcal{A}$.

\renewcommand{\arraystretch}{1.2}
\begin{table}
\footnotesize
$$\begin{array}{llllll} \hline
S & |x_1| & n_1 & |x_2| & n_2 & \mathcal{M}(x_1) \\ \hline
{}^2B_2(q),\, q \geqs 2^7 & q+\sqrt{2q}+1 & 4 & q-\sqrt{2q}+1 & 4 & \la x_1 \ra{:}\Z_4 \\
{}^2G_2(q),\, q \geqs 3^5 & q+\sqrt{3q}+1 & 6 & q-\sqrt{3q}+1 & 6 & \la x_1 \ra{:}\Z_6 \\
{}^2F_4(q),\, q \geqs 2^3 & q^2+\sqrt{2q^3}+q+\sqrt{2q}+1 & 12 & q^2-\sqrt{2q^3}+q-\sqrt{2q}+1 & 12 & \la x_1 \ra{:}\Z_{12} \\
{}^3D_4(q),\, q \geqs 3 & q^4-q^2+1 & 4 & (q^3-1)(q+1) & 4 & \la x_1 \ra{:}\Z_4 \\
{}^2E_6(q), \, q \geqs 4 & (q^6-q^3+1)/a & 9 & (q^5+1)(q-1)/a & 10 & {\rm SU}_{3}(q^3).3 \\
G_2(q),\, q \geqs 9 &  q^2-q+1 & 6 & q^2+q+1 & 6 & {\rm SU}_{3}(q).2 \\
F_4(q),\, q \geqs 4 & q^4-q^2+1 & 12 & q^4+1 & 8 & {}^3D_4(q).3 \\
E_6(q) & (q^6+q^3+1)/b & 9 & (q+1)(q^5-1)/b & 10 & {\rm SL}_{3}(q^3).3 \\
E_7(q),\, q \geqs 4 & (q+1)(q^6-q^3+1)/c & 18 & (q-1)(q^6+q^3+1)/c & 18 & (\Z_{(q+1)/c}.{}^2E_6(q)).2 \\
E_8(q) & q^8+q^7-q^5-q^4-q^3+q+1 & 30 & q^8-q^7+q^5-q^4+q^3-q+1 & 30 & \la x_1 \ra{:}\Z_{30} \\ \hline
\multicolumn{6}{l}{\mbox{\tiny $a=(3,q+1)$, $b=(3,q-1)$, $c=(2,q-1)$}} \\
\end{array}$$
\caption{}
\label{tab:ex}
\end{table}
\renewcommand{\arraystretch}{1}

\begin{proposition}\label{p:main_ex}
The conclusion to Theorem \ref{t:main2} holds if $S$ is one of the simple groups listed in Table \ref{tab:ex}.
\end{proposition}

\begin{proof}

Let $S$ be one of the simple groups listed in Table \ref{tab:ex}. First we claim that there exist
elements $x_1,x_2 \in S$ with the following properties:
\begin{itemize}\addtolength{\itemsep}{0.2\baselineskip}
\item[(i)] $x_1$ and $x_2$ are self-centralising;
\item[(ii)] $|x_i|$ (the order of $x_i$) and $|\Normalizer_S(\la x_i \ra): \la x_i \ra|=n_i$ are given in Table \ref{tab:ex};
\item[(iii)] Let $\mathcal{M}(x_1)$ be the set of maximal subgroups of $S$ containing $x_1$, up to isomorphism.
Then $\mathcal{M}(x_1)$ is given in the final column of Table \ref{tab:ex}.
\end{itemize}

Detailed information on the conjugacy classes in $S$ is readily available in the literature, and
the existence and self-centralising nature of $x_1$ and $x_2$ can be quickly verified. In each
case, $\la x_i \ra$ is a maximal torus of $S$ and the indices $n_i$ are easily computed. Indeed, if
$S \neq E_7(q)$ then $n_1$ is given in \cite[Table 1]{BPS}, and the same table also records  $n_2$
in the cases $S \in \{{}^2B_2(q),{}^2G_2(q),{}^2F_4(q),E_8(q)\}$. If $S={}^3D_4(q)$ then $n_2$ is
given in \cite[Table 1.1]{DMi}. In the remaining cases we have $S=E_6^{\e}(q)$ or $E_7(q)$, and the
$n_i$ can be read off from \cite{FJ}. More precisely, if $S=E_6^{\e}(q)$ then $x_2$ corresponds to
the case labelled $w16$ on \cite[p.103]{FJ}, where $n_2$ is denoted ``cn". Similarly, if $S=E_7(q)$
then $x_1$ and $x_2$ are the cases labelled $w56$ and $w47$ on \cite[p.134,135]{FJ}, respectively.
Finally, the information on $\mathcal{M}(x_1)$ is taken from \cite[Section 4]{Weigel} (see also
\cite[Table III]{GK}).

The argument in each case is very similar. For example, suppose $S=E_6(q)$. Define $x_i,n_i,b$ as
in Table \ref{tab:ex} and note that $|{\rm Out}(S)| = 2b\log_pq$. Let $H$ be a maximal subgroup of
$S$ and recall that it suffices to show that there are at least three $A$-classes in $S$ that fail
to meet $H$ (and that the number of such classes tends to infinity as $|S|$ tends to infinity).

Set $\a_i = |x_i|$ and let $a_i$ be the number of distinct $A$-classes of elements in $S$ of order
$\a_i$. By Lemmas \ref{l:euler} and \ref{l:class} we have
$$a_1 \geqs \left\lceil \frac{\phi(\a_1)}{18b\log_p q} \right\rceil \geqs \frac{\sqrt{\a_1}}{18b\log_pq},$$
so $a_1 \geqs 3$, and we observe that $a_1$ tends to infinity as $q$ tends to infinity. Now $x_1$
belongs to a unique maximal subgroup of $S$, which is isomorphic to ${\rm SL}_{3}(q^3).3$ (see
\cite[p.78--79]{Weigel}). Therefore, it remains to deal with the case $H = {\rm SL}_{3}(q^3).3$.
Since $|H|$ is indivisible by $\a_2$, it follows that any element of order $\a_2$ is a derangement,
and as before we deduce that
$$a_2 \geqs \left\lceil \frac{\phi(\a_2)}{20b\log_p q} \right\rceil \geqs \frac{\sqrt{\a_2/2}}{20b\log_pq}.$$
The result follows. The other cases are entirely similar, and we omit the details.
\end{proof}

\begin{proposition}\label{p:ex2}
The conclusion to Theorem \ref{t:main2} holds if $S$ is an exceptional group of Lie type.
\end{proposition}

\begin{proof}

We may assume that $S \in \mathcal{B}$, where $\mathcal{B}$ is defined as follows:
$$\mathcal{B} = \{G_2(7), G_2(8), {}^2E_6(2), {}^2E_6(3), F_4(2), F_4(3), E_7(2), E_7(3)\}.$$
If $S \in \{G_2(7), {}^2E_6(3), F_4(3), E_7(3)\}$ then the argument in the proof of Proposition
\ref{p:main_ex} goes through unchanged (see \cite[Table IV]{GK}). In each of the remaining cases,
we define $x_i,n_i,\a_i,a_i$ as before. Let $H$ be a maximal subgroup of $S$.

Suppose $S=G_2(8)$, so $\a_1=57$ and $\a_2=73$. Since $a_2 \geqs \phi(\a_2)/18 = 4$ we may assume
that $H = {\rm SL}_{3}(8).2$ since no other maximal subgroups of $S$ contain elements of order $73$
(the maximal subgroups of $S$ are determined in \cite{Coop}). Since $a_1 \geqs 2$ and $|H|$ is
indivisible by $\a_1$ and $19$, the result follows.

Next suppose $S={}^2E_6(2)$. Note that the list of maximal subgroups $H$ of $S$ given in the Atlas
\cite{ATLAS} is complete (see \cite[p.304]{ModAt}). If $|\pi(S) \setminus \pi(H)| \geqs 3$ then we
are done, so we may assume that $H \in \{F_4(2), {\rm Fi}_{22}, \O_{10}^{-}(2)\}$.
In each of these cases, the fusion of $H$-classes in $S$ is available in the \textsf{GAPCTL}
Character Table Library \cite{GAPCTL}, and the desired result quickly follows.

The case $S=F_4(2)$ is very similar. Here the maximal subgroups of $S$ are determined in \cite{NW}.
If $H=(2^{1+8} \times 2^6).{\rm Sp}_{6}(2)$ or ${\rm Sp}_{8}(2)$ then the fusion of $H$-classes in
$S$ is stored in \cite{GAPCTL} and we easily deduce that $\kappa = 12, 33$ in these cases. If
$|\pi(S) \setminus \pi(H)| \geqs 3$ then the result follows, so it remains to deal with the cases
$H \in \{{\rm L}_{4}(3){:}2,{}^2F_4(2),{}^3D_4(2){:}3,\O_{8}^{+}(2){:}S_3\}$. In all four cases it
is easy to check that $H$ contains no elements of order $32$, but there are four $A$-classes of
such elements, so $\kappa \geqs 4$ in each of these cases.

Finally, let us assume $S=E_7(2)$. Following \cite[Table IV]{GK}, let $x \in S$ be an element of
order $\a = 2^7+1 = 129$. Then $\Centralizer_S(x) = \la x \ra$ and $|\Normalizer_S(\la x \ra):\la x \ra| = 14$ (see the
case labelled $w57$ in \cite[p.120]{FJ}), so Lemma \ref{l:class} implies that there are at least
$\phi(\a)/14 = 6$ distinct $A$-classes of such elements. Moreover, \cite[Table IV]{GK} indicates
that $x$ is contained in a unique maximal subgroup ${\rm SU}_{8}(2)$ of $S$. Therefore, we may
assume that $H={\rm SU}_{8}(2)$. Now $|H|$ is indivisible by $\b=2^7-1 = 127$ and we calculate that
there are at least $\phi(127)/14 = 9$ distinct $A$-classes of elements of order $127$. The result follows.
\end{proof}

\subsection{Classical groups}\label{ss:class}

In order to complete the proof of Theorem \ref{t:main2}, we may assume that $S$ is one of the
classical groups listed in Table \ref{tab:cl}. The conditions recorded in the final column  ensure
that $S$ is simple, and that $S$ is not isomorphic to one of the other groups in the table, or to
one of the groups we have already considered (see \cite[Proposition 2.9.1]{KL}, for example). We
will write ${\rm L}_{n}^{+}(q) = {\rm L}_{n}(q)$ and ${\rm L}_{n}^{-}(q) = {\rm U}_{n}(q)$ to
denote ${\rm PSL}_{n}(q)$ and ${\rm PSU}_{n}(q)$, respectively. Let $V$ be the natural $S$-module
and set $A={\rm Aut}(S)$.

\begin{table}
\renewcommand{\arraystretch}{1.2}
\begin{tabular}{lll} \hline
 & $S$ & Conditions \\ \hline
Linear & ${\rm L}_{n}(q)$ &  $n \geqs 2$, $q \geqs 7$ ($q \neq 9$) if $n=2$, $(n,q) \neq (3,2), (4,2)$ \\
Unitary & ${\rm U}_{n}(q)$ & $n \geqs 3$, $(n,q) \neq (3,2)$ \\
Symplectic & ${\rm PSp}_{2m}(q)$ & $m \geqs 2$, $(m,q) \neq (2,2), (2,3)$ \\
Orthogonal & ${\rm P\O}_{2m}^{\pm}(q)$ & $m \geqs 4$ \\
& $\O_{2m+1}(q)$ & $m \geqs 3$, $q$ odd \\ \hline
\end{tabular}
\caption{The simple classical groups}
\label{tab:cl}
\end{table}

As before, it suffices to show that if $H$ is a maximal subgroup of $S$ then there are at least
three $A$-classes of elements $x \in S$ such that $x^A \cap H$ is empty (and that the number of
such $A$-classes tends to infinity as $|S|$ tends to infinity). As in the previous section, we will
identify a sufficient number of $A$-classes of elements that belong to a very restricted collection
of maximal subgroups (in almost all cases, these will be regular semisimple elements). In order to
do this, we will use several results from \cite{BGK,GK}, which rely on the earlier analysis of
primitive prime divisors in \cite{GPPS}. It then remains to deal with the primitive groups that
correspond to this very specific list of maximal subgroups, and we will identify an
alternative collection of $A$-classes of derangements. As before, it is convenient to use {\sc
Magma} for some low-dimensional groups over small fields.

\subsubsection{Linear and unitary groups}

Here we assume $S={\rm L}_{n}^{\e}(q)$. Set $d=(n,q-\e)$ and $e=d(q-\e)$.

\begin{proposition}\label{p:lucomp}
The conclusion to Theorem \ref{t:main2} holds if $S$ is one of the following:
$${\rm L}_{2}(q), \, q \leqs 81; \; {\rm L}_{3}(q), \, q \leqs 16; \,
{\rm L}_{4}(q), \, q \leqs 9; \, {\rm L}_{5}(2);\, {\rm L}_{7}(2); \, {\rm L}_{11}(2)$$
$${\rm U}_{3}(q),\, q \leqs 11;\, {\rm U}_{4}(q),\, q \leqs 7; \,
{\rm U}_{5}(2);\, {\rm U}_{6}(2);\, {\rm U}_{8}(2);\,{\rm U}_{8}(3);\, {\rm U}_{9}(2);\, {\rm U}_{12}(2).$$
\end{proposition}

\begin{proof}

This is a straightforward verification. For example, suppose $S = {\rm L}_{2}(q)$. If $16<q \leqs
81$ then an easy {\sc Magma} calculation shows that if $H$ is a maximal subgroup of $S$ then there
are at least three $A$-classes of elements in $S$ that fail to meet $H$. If $7 \leqs q \leqs 16$
then we consider each possibility for $G$ in turn, using {\sc Magma} to compute $\kappa(G,H)$ for
each maximal subgroup $H$ of $G$ (with $G = SH$). The other linear groups with $n \leqs 5$ are
handled in the same way. If $S = {\rm L}_{7}(2)$ then any element of order $2^{7}-1$ is a
derangement, unless $H$ is a field extension subgroup of type ${\rm GL}_{1}(2^{7})$, in which case elements of order $2^{6}-1$ are derangements. The case $S =
{\rm L}_{11}(2)$ is entirely similar.

The argument for the unitary groups ${\rm U}_{n}(q)$ with $n<8$ is similar. If $S = {\rm U}_{8}(q)$
then \cite[Proposition 5.22]{BGK} implies that any element of order $(q^7+1)/d$ is a derangement
unless $H$ is a $P_1$ parabolic subgroup (the stabiliser of a totally singular $1$-space), in which
case we can take any element of order $(q^5+1)(q^3+1)/e$. Similarly, if $S={\rm U}_{12}(2)$ then by
considering elements of order $(2^{11}+1)/3$ we reduce to the case $H=P_1$, for which all elements of
order $(2^7+1)(2^5+1)/9$ are derangements. Finally, if $S={\rm U}_{9}(2)$ then any element of order
$(2^9+1)/9$ is a derangement unless $H$ is a field extension subgroup of type ${\rm GU}_{3}(8)$. In
the latter case, elements of order $(2^8-1)/3$ are derangements.
\end{proof}

\begin{proposition}\label{p:lu1}
The conclusion to Theorem \ref{t:main2} holds if $S={\rm L}_{n}^{\e}(q)$ is one of the groups listed in Table \ref{tab:class1}.
\end{proposition}

\begin{proof}
This is similar to the proof of Proposition \ref{p:main_ex}.
We may assume that $S$ is not one of the groups in the statement of Proposition \ref{p:lucomp}. There are several cases to consider.

First assume $\e=+$ and $n=2m$ is even, where $m \geqs 3$ is odd. Let $T = \la x_1 \ra$ be a cyclic
maximal torus of $S$ of order $\a_1 = (q^{m+2}-1)(q^{m-2}-1)/e$ (see \cite[Theorem 2.1]{BG}, for
example), so $x_1$ is self-centralising and $|\Normalizer_S(\la x_1 \ra):\la x_1 \ra| = m^2-4$.
Let $a_1$ be the number of distinct $A$-classes of elements in $S$ of order $\a_1$. By applying Lemmas \ref{l:euler} and \ref{l:class}, we deduce that
$$a_1 \geqs \left\lceil \frac{\phi(\a_1)}{(m^2-4) \cdot 2d\log_pq} \right\rceil \geqs \frac{\sqrt{\a_1/2}}{2d(m^2-4)\log_pq}.$$
It follows that $a_1 \geqs 3$, and we also see that $a_1$ tends to infinity as $|S|$ tends to infinity.

Now $x_1$ belongs to exactly two maximal subgroups of $S$; parabolic subgroups of type $P_{m-2}$
and $P_{m+2}$ (see \cite[Table II]{GK}). Therefore, in order to establish Theorem \ref{t:main2}
in this case, we may assume that $H=P_{m-2}$. Let $x_2 \in S$ be an element of order $\a_2 =
(q^{n}-1)/e$. Then $x_2$ is self-centralising, $|\Normalizer_S(\la x_2 \ra):\la x_2 \ra| = n$ and $x_2$ is a
derangement since it acts irreducibly on $V$. If $a_2$ is the number of $A$-classes of such
elements then
$$a_2 \geqs \left\lceil \frac{\phi(\a_2)}{2m \cdot 2d\log_pq} \right\rceil \geqs \frac{\sqrt{\a_2/2}}{4md\log_pq}$$
and the result follows.

The other cases in Table \ref{tab:class1} are very similar. In each case we take $x_1 \in S$ of the
given order, noting that $x_1$ is self-centralising and $|\Normalizer_S(\la x_1 \ra):\la x_1 \ra| = n_1$. As
above, we estimate the number of $A$-classes of such elements, and we appeal to \cite[Table II]{GK}
to see that the only maximal subgroups of $S$ containing $x_1$ are reducible. To complete the
proof, we now switch to the self-centralising elements $x_2$, as indicated in Table
\ref{tab:class1}, and we repeat the above argument.
\end{proof}

\renewcommand{\arraystretch}{1.2}
\begin{table}
\footnotesize
$$\begin{array}{lllll} \hline
\mbox{Conditions on $S={\rm L}_{n}^{\e}(q)$} & |x_1| & n_1 & |x_2| & n_2 \\ \hline
\mbox{$n=2m$, $m \geqs 4-\e$ odd} & (q^{m+2}-\e)(q^{m-2}-\e)/e & m^2-4 & (q^{n}-1)/e & n \\
\mbox{$n=2m$, $m \geqs 3-\e$ even,} & (q^{m+1}-\e)(q^{m-1}-\e)/e & m^2-1 & (q^{n}-1)/e & n \\
(\e,m,q) \neq (+,2,2) & & & & \\
\mbox{$n=2m+1$, $m \geqs 2$, $\e=+$,} & (q^{m+1}-1)(q^{m}-1)/e & m(m+1) & (q^{n}-1)/e & n \\
(m,q) \neq (5,2) & & & & \\
\mbox{$n=2m+1$, $m \geqs 5$ odd, $\e=-$} & (q^{m+2}+1)(q^{m-1}-1)/e & (m+2)(m-1) & (q^{n}+1)/e & n \\
\mbox{$n=2m+1$, $m \geqs 4$ even, $\e=-$} & (q^{m+1}+1)(q^{m}-1)/e & m(m+1) & (q^{n}+1)/e & n \\ \hline
\multicolumn{5}{l}{\mbox{\tiny $e=(q-\e)(n,q-\e)$}} \\
\end{array}$$
\caption{}
\label{tab:class1}
\end{table}

\begin{proposition}\label{p:lu2}
The conclusion to Theorem \ref{t:main2} holds if $S$ is a linear or unitary group.
\end{proposition}

\begin{proof}

It remains to deal with the possibilities for $S$ listed in Table \ref{tab:class2}, and we proceed
as in the proof of the previous proposition. For example, suppose $S={\rm L}_{2}(q)$. Let $x_1 \in
S$ be an element of order $\a_1=(q^2-1)/e$. Then $x_1$ is self-centralising and the number of
distinct $A$-classes of such elements, denoted by $a_1$, satisfies the bound
$$a_1 \geqs \left\lceil \frac{\phi(\a_1)}{2 \cdot d\log_pq} \right\rceil \geqs \frac{\sqrt{\a_1/2}}{2d\log_pq}.$$
In particular, $a_1$ tends to infinity as $|S|$ tends to infinity, and we calculate that $a_1 \geqs
3$ since $q \geqs 83$. Now $x_1$ belongs to a unique maximal subgroup of $S$, namely
$H=\Normalizer_S(\la x_1 \ra)$ (see \cite[p.767]{GK}). Let $x_2 \in S$ be an element of order $\a_2
= (q-1)^2/e$ and let $a_2$ be the number of $A$-classes of such elements. Note that $x_2$ is a
derangement since $|H|$ is indivisible by $\a_2$. Then
$$a_2 \geqs  \left\lceil \frac{\phi(\a_2)}{2 \cdot d\log_pq} \right\rceil \geqs \frac{\sqrt{\a_2/2}}{2d\log_pq}$$
and thus $a_2 \geqs 3$ if $q>125$. In fact, if $81<q \leqs 125$ then an easy {\sc Magma} calculation shows that $a_2 \geqs 4$. The result follows.

The other cases in Table \ref{tab:class2} are handled in the same way, using the information in
\cite[p.767]{GK} (in each case, note that $|x_1|$ and $|x_2|$ are coprime). We omit the details.
\end{proof}

\renewcommand{\arraystretch}{1.2}
\begin{table}
\footnotesize
$$\begin{array}{lllll} \hline
S & |x_1| & n_1 & |x_2| & n_2 \\ \hline
{\rm L}_{2}(q), q \geqs 83 & (q^2-1)/e & 2 & (q-1)^2/e & 2 \\
{\rm L}_{3}(q), q \geqs 17 & (q^3-1)/e & 3 & (q^2-1)(q-1)/e & 2 \\
{\rm U}_{3}(q), q \geqs 13 & (q^3+1)/e & 3 & (q^2-1)(q+1)/e & 2 \\
{\rm U}_{4}(q), q \geqs 8 & (q^3+1)(q+1)/e & 3 & (q^4-1)/e & 4 \\
{\rm U}_{5}(q), q \geqs 3 &  (q^5+1)/e & 5 & (q^4-1)(q+1)/e & 4 \\
{\rm U}_{6}(q), q \geqs 3 &  (q^5+1)(q+1)/e & 5 & (q^6-1)/e & 6 \\
{\rm U}_{7}(q) &  (q^7+1)/e & 7 & (q^6-1)(q+1)/e & 6 \\ \hline
\multicolumn{5}{l}{\mbox{\tiny $e=(q-\e)(n,q-\e)$}} \\
\end{array}$$
\caption{}
\label{tab:class2}
\end{table}

\subsection{Symplectic groups}

Here we assume $S={\rm PSp}_{2m}(q)$, where $m \geqs 2$ and $(m,q) \neq (2,2),(2,3)$ (since ${\rm
PSp}_{4}(2)' \cong A_6$ and ${\rm PSp}_{4}(3) \cong {\rm U}_{4}(2)$). Set $d=(2,q-1)$.

We will frequently refer to the regular semisimple elements $x_i \in S$ defined in Table
\ref{tab:class3}. In the second column, we give an orthogonal decomposition of the natural
$S$-module $V$ that is fixed by $x_i$, with $x_i$ acting irreducibly on each nondegenerate subspace
in the decomposition (the same notation is used in \cite{BGK,GK}). The order $\a_i$ of a lift
$\hat{x}_i \in {\rm Sp}_{2m}(q)$ of $x_i$ is given in the next column, and in the final column we
record a lower bound $a_i \geqs \b_i$, where $a_i$ denotes the number of $A$-classes of elements in
$S$ with the same shape and order as $x_i$ (the lower bound follows from the fact that two
semisimple elements in ${\rm Sp}_{2m}(q)$ are conjugate if and only if they have the same multiset
of eigenvalues in $\bar{\mathbb{F}}_{q}$).

\begin{proposition}\label{p:sp0}
The conclusion to Theorem \ref{t:main2} holds if $S$ is one of the following:
$${\rm PSp}_{4}(q),\, q \leqs 8;\; {\rm PSp}_{6}(q), \, q \leqs 3;\; {\rm Sp}_{8}(2); \; {\rm Sp}_{10}(2);
\;{\rm Sp}_{12}(2); \; {\rm Sp}_{14}(2).$$
\end{proposition}

\begin{proof}
Set $S = {\rm PSp}_{2m}(q)$. In the cases with $m \leqs 5$ we can use {\sc Magma} to compute
$\kappa(G,H)$ for every maximal subgroup $H$ of $G$, and the result quickly follows. Now assume
$(m,q) = (6,2)$ or $(7,2)$. Consider the irreducible elements of type $x_1$ defined in Table
\ref{tab:class3}, and note that the bound $a_1 \geqs \b_1$ implies that $a_1 \geqs 3$. By the proof
of \cite[Proposition 5.8]{BGK}, we may assume that $H$ is of type ${\rm O}_{2m}^{-}(q)$ or ${\rm
Sp}_{2m/k}(q^k)$, where $k$ is a prime divisor of $m$ (these are the only maximal subgroups of $S$
that contain such elements). In both cases we observe that semisimple elements of type $x_2$, and
regular unipotent elements (that is, unipotent elements with Jordan form $[J_{2m}]$) are
derangements. The result now follows since the bound $a_2 \geqs \b_2$ in Table \ref{tab:class3}
implies that $a_2 \geqs 2$.
\end{proof}

\renewcommand{\arraystretch}{1.2}
\begin{table}
\footnotesize
$$\begin{array}{llll} \hline
& \mbox{Decomposition} & \a_i & \b_i \\ \hline
x_1 & 2m & q^m+1 & \phi(q^m+1)/2me \\
x_2 & 2 \perp (2m-2) & {\rm lcm}(q+1,q^{m-1}+1) & \phi(q+1)\phi(q^{m-1}+1)/2(2m-2)e \\
x_3 & 4 \perp (2m-4) & {\rm lcm}(q^2+1,q^{m-2}+1) & \phi(q^2+1)\phi(q^{m-2}+1)/4(2m-4)e \\ \hline
\multicolumn{4}{l}{\mbox{\tiny $d=(2,q-1)$, $e=2^{\delta}d\log_pq$, $\delta=1$ if $m=p=2$ and $\delta=0$ otherwise}} \\
\end{array}$$
\caption{}
\label{tab:class3}
\end{table}

\begin{proposition}\label{p:sp1}
The conclusion to Theorem \ref{t:main2} holds if $S={\rm PSp}_{2m}(q)$ and $m \geqs 5$.
\end{proposition}

\begin{proof}

Let $H$ be a maximal subgroup of $S$. It suffices to show that there are at least three $A$-classes
of elements $x \in S$ such that $x^A \cap H$ is empty (and that the number of such $A$-classes
tends to infinity as $|S|$ tends to infinity).  We will assume that $S$ is not one of the groups in
the statement of Proposition \ref{p:sp0}. We continue to adopt the notation introduced in Table
\ref{tab:class3}. It is important to note that the bounds $a_i \geqs \b_i$ in Table
\ref{tab:class3}, together with the conditions on $m$ and $q$, imply that $a_i \geqs 3$ in all
cases (with the exception of $a_3$ if $(m,q)=(5,3)$), and it is clear that $a_i$ tends to infinity
as $|S|$ tends to infinity.

First assume $mq$ is odd. Consider elements of type $x_2$, as described in Table \ref{tab:class3}.
According to \cite[Proposition 5.10]{BGK}, the only maximal subgroup of $S$ containing such an
element is the stabiliser of a nondegenerate $2$-space, denoted by $N_2$. Since $a_2 \geqs 3$ (and
$a_2$ tends to infinity as $|S|$ tends to infinity), we have reduced to the case $H=N_2$. In this
situation, irreducible elements of type $x_1$ are derangements and the result follows.

Next suppose $q$ is odd and $m \geqs 6$ is even. Here we use elements of type $x_3$ to reduce to
the case where $H$ is either a subspace subgroup of type $N_4$ (the stabiliser of a nondegenerate
$4$-space), or a field extension subgroup of type ${\rm Sp}_{m}(q^2)$ (see \cite[Proposition
5.10]{BGK}). These subgroups can be handled as before, using elements of type $x_1$ and $x_2$,
respectively (note that the order of the field extension subgroup is indivisible by $|x_2|$).

Finally, let us assume $q$ is even. By considering elements of type $x_1$, and by inspecting the
proof of \cite[Proposition 5.8]{BGK}, we reduce to the case where $H$ is of type ${\rm
O}_{2m}^{-}(q)$ or ${\rm Sp}_{2m/k}(q^k)$ for a prime divisor $k$ of $m$. Now $x_2$ fixes an
orthogonal decomposition of the form $2 \perp (2m-2)$, which implies that $x_2 \in {\rm
O}_{2}^{-}(q) \times {\rm O}_{2m-2}^{-}(q) < {\rm O}_{2m}^{+}(q)$ and thus $x_2$ is a derangement
if $H$ is of type ${\rm O}_{2m}^{-}(q)$. Since $|x_2|$ does not divide the order of a field
extension subgroup, we also deduce that these elements are derangements if $H$ is of type ${\rm
Sp}_{2m/k}(q^k)$. The result follows.
\end{proof}

\begin{proposition}\label{p:sp2}
The conclusion to Theorem \ref{t:main2} holds if $S$ is a symplectic group.
\end{proposition}

\begin{proof}

We may assume that $2 \leqs m \leqs 4$. We may also assume that $S$ is not one of the cases handled
in Proposition \ref{p:sp0}. We continue to adopt the notation introduced in Table \ref{tab:class3}.
In particular, we set $d=(2,q-1)$ and $e=2^{\delta}d\log_pq$, where $\delta=1$ if $m=p=2$ and
$\delta=0$ otherwise. As usual, let $H$ be a maximal subgroup of $S$.

First assume $q$ is odd and note that the bound $a_1 \geqs \b_1$ in Table \ref{tab:class3} implies
that $a_1 \geqs 3$. If $m=2$ or $4$ then by considering elements of type $x_1$ we reduce to the
case where $H$ is of type ${\rm Sp}_{m}(q^2)$ (see \cite[Proposition 5.12]{BGK}). In this
situation, we define an element $x_4 \in S$ that fixes an orthogonal decomposition $2 \perp (2m-2)$
of $V$ by centralising the $2$-space and acting irreducibly on the $(2m-2)$-space. Then $x_4$ is a
derangement and we note that
\begin{equation}\label{e:a4}
a_4 \geqs \left\lceil \frac{\phi(q^{m-1}+1)}{(2m-2)e}\right\rceil
\end{equation}
where $a_4$ denotes the number of $A$-classes of such elements. In particular, it follows that $a_4
\geqs 3$ if $m=2$ and $q>29$, or if $m=4$ and $q>5$. (We also note that $a_4$ tends to infinity as
$|S|$ tends to infinity.) Of course, unipotent elements with Jordan form $[J_2,J_1^{2m-2}]$ or
$[J_4,J_{1}^{2m-4}]$ are also derangements (where $J_i$ denotes a standard unipotent Jordan block
of size $i$), so it is easy to see that there are always at least three distinct $A$-classes of
derangements.

Similarly, if $q$ is odd and $m=3$ then we reduce to subgroups of type ${\rm GU}_{3}(q)$ and ${\rm
Sp}_{2}(q^3)$ via elements of type $x_1$ (see \cite[Main Theorem]{Ber}, for example). In these
cases, elements of type $x_2$ are derangements, and the result follows since $a_2 \geqs 3$ (if
$q>5$ then this follows from the bound $a_2 \geqs \b_2$, and for $q=5$ it can be checked directly).

Finally, suppose $q$ is even. In the usual manner, by considering elements of type $x_1$ and
applying \cite[Proposition 5.8]{BGK}, we reduce to the case where $H$ is of type ${\rm
O}_{2m}^{-}(q)$ or ${\rm Sp}_{2m/k}(q^k)$, with $k$ a prime divisor of $m$. In the first case,
elements of type $x_2$ are derangements and the result follows. Similarly, if $H$ is of type ${\rm
Sp}_{2m/k}(q^k)$ then elements of type $x_4$ (as defined above) are derangements, and the result
follows via the lower bound in \eqref{e:a4} (and the fact that unipotent elements with Jordan form
$[J_2,J_1^{2m-2}]$ or $[J_4,J_{1}^{2m-4}]$ are also derangements).
\end{proof}

\subsection{Orthogonal groups}

Finally, let us assume $S={\rm P\O}_{n}^{\e}(q)$, where $n \geqs 7$. Set $A = {\rm Aut}(S)$ and
define $d=(2,q-1)$ if $n$ is even, and $d=1$ if $n$ is odd. As in the previous section, we will
denote an orthogonal decomposition $V=U \perp W$ with $\dim U = m$ by writing $m \perp (n-m)$. If
$m$ is even, in order to distinguish between nondegenerate $m$-spaces of plus and minus types, we
will write $m^{+}$ and $m^{-}$, respectively. This is consistent with the notation used in
\cite{BGK,GK}.

\renewcommand{\arraystretch}{1.2}
\begin{table}
\footnotesize
$$\begin{array}{lll} \hline
& \mbox{Decomposition}  & \b_i \\ \hline
x_1 & 2^{-} \perp (2m-2)^{-}
& \phi(q+1)\phi(q^{m-1}+1)/2(2m-2)e \\
x_2,\, \mbox{$m$ odd} & (m-1)^{-} \perp (m+1)^{-}
& \phi(q^{(m-1)/2}+1)\phi(q^{(m+1)/2}+1)/(m^2-1)e \\
x_3,\, \mbox{$m$ even} & (m-2)^{-} \perp (m+2)^{-}
& \phi(q^{(m-2)/2}+1)\phi(q^{(m+2)/2}+1)/(m^2-4)e \\ \hline
\multicolumn{3}{l}{\mbox{\tiny $e=2d\log_pq$}} \\
\end{array}$$
\caption{}
\label{tab:class4}
\end{table}

\begin{proposition}\label{p:ort0}
The conclusion to Theorem \ref{t:main2} holds if $S$ is one of the following:
$${\rm P\O}_{8}^{\e}(q),\, q \leqs 4;\; {\rm P\O}_{10}^{\e}(q),\, q \leqs 3;\; {\rm P\O}_{12}^{\e}(q), \,
q \leqs 3;\; \O_{14}^{+}(2); \; \O_{16}^{+}(2);\; \O_{18}^{+}(2).$$
\end{proposition}

\begin{proof}

Set $S = {\rm P\O}_{2m}^{\e}(q)$. If $m \leqs 4$ or $(m,q) \in \{(5,2),(6,2)\}$ then the result can
be checked using {\sc Magma}.

Next suppose $S = \O_{14}^{+}(2)$. Let $x \in S$ be an element of order $195$ that fixes an
orthogonal decomposition $2^{-} \perp 12^{-}$ of the natural $S$-module. Using {\sc Magma}, we see
that there are at least three $A$-classes of such elements, and by the main theorem of \cite{GPPS}
we deduce that the only maximal subgroup of $S$ containing $x$ is the stabiliser of a nondegenerate
$2$-space of minus-type, which we denote by $N_{2}^{-}$. Therefore, we have reduced to the case
$H=N_{2}^{-}$. It is easy to identify three classes of derangements in this case. For instance, any
unipotent element with Jordan form $[J_{13},J_{1}]$, $[J_{11},J_{3}]$ or $[J_{9},J_{5}]$ is a
derangement. The cases $S=\O_{16}^{+}(2)$ and $\O_{18}^{+}(2)$ are entirely similar.

It remains to deal with the cases $S={\rm P\O}^{\e}_{10}(3)$ and ${\rm P\O}^{\e}_{12}(3)$. First
assume $S={\rm P\O}_{10}^{+}(3)$. Let $x \in S$ be an element of order $82$ that fixes an
orthogonal decomposition $2^{-}\perp 8^{-}$. There are at least three $A$-classes of such elements,
and the main theorem of \cite{GPPS} implies that the only maximal subgroups of $S$ that contain
such elements are of type $N_{2}^{-}$ or ${\rm O}_{5}(9)$. The result now follows because it is
easy to see that there are at least three $A$-classes of derangements if $H$ is of type $N_{2}^{-}$
or ${\rm O}_{5}(9)$; for example, any unipotent element with Jordan form $[J_9,J_1]$, $[J_7,J_3]$
or $[J_5,J_2^2,J_1]$ is a derangement. The case $S={\rm P\O}_{12}^{+}(3)$ is very similar (working
with elements of order $122$ that fix a decomposition $2^{-}\perp 10^{-}$).

The cases $S={\rm P\O}_{10}^{-}(3)$ or ${\rm P\O}_{12}^{-}(3)$ are also similar. If $S={\rm
P\O}_{10}^{-}(3)$ then the only maximal subgroups of $S$ that contain elements of order $61$ are of
type ${\rm GU}_{5}(3)$, and the result follows as before. Similarly, if $S={\rm P\O}_{12}^{-}(3)$
then we work with elements of order $365$, which only belong to maximal subgroups of type ${\rm
O}_{6}^{-}(3^2)$ or ${\rm O}_{4}^{-}(3^3)$. Again, the result quickly follows.
\end{proof}

\begin{proposition}\label{p:plus}
The conclusion to Theorem \ref{t:main2} holds if $S={\rm P\O}_{2m}^{+}(q)$ and $m \geqs 5$.
\end{proposition}

\begin{proof}

We may assume that $S$ is not one of the groups in the statement of Proposition \ref{p:ort0}. We
define the regular semisimple elements $x_i \in S$ as in Table \ref{tab:class4} (we use the same
notation as in \cite{BGK,GK}), where $\b_i$ is a lower bound on $a_i$, which is the number of
distinct $A$-classes of elements in $S$ with the same shape and order as $x_i$. Let $H$ be a
maximal subgroup of $S$.

First assume $m$ is odd. Consider elements of type $x_2$. By \cite[Proposition 5.13]{BGK}, the only
maximal subgroups of $S$ containing $x_2$ are of type $N_{m-1}^{-}$. The lower bound $a_2 \geqs
\b_2$ implies that $a_2 \geqs 3$ (and that $a_2$ tends to infinity as $|S|$ tends to infinity).
Now, if $H=N_{m-1}^{-}$ then elements of type $x_1$ are derangements, and we observe that the lower
bound $a_1 \geqs \b_1$ is sufficient.

Now assume $m$ is even. By considering elements of type $x_3$, and by applying \cite[Proposition
5.14]{BGK}, we reduce to the case where $H$ is of type $N_{m-2}^{-}$ or ${\rm O}_{m}^{+}(q^2)$.
Here elements of type $x_1$ are derangements, and the result follows via the bound $a_1 \geqs
\b_1$.
\end{proof}

\begin{proposition}\label{p:plus8}
The conclusion to Theorem \ref{t:main2} holds if $S={\rm P\O}_{8}^{+}(q)$.
\end{proposition}

\begin{proof}
In view of Proposition \ref{p:ort0}, we may assume that $q \geqs 5$. Let $x_1 \in S$ be a regular
semisimple element that fixes an orthogonal decomposition $2^{-}\perp 6^{-}$. Let $a_1$ denote the
number of distinct $A$-classes of such elements. Then
$$a_1 \geqs \left\lceil \frac{\phi(q+1)\phi(q^3+1)}{12\cdot 6d\log_pq}\right\rceil$$
and we deduce that $a_1 \geqs 3$ if $q \geqs 7$ (and that $a_1$ tends to infinity as $|S|$ tends to
infinity). If $q=5$ then a direct calculation shows that there are at least three $A$-classes of
such elements (in particular, none of the relevant ${\rm PGO}_{8}^{+}(5)$-classes are fused by a
triality graph automorphism of $S$). By \cite[p.767]{GK}, the only maximal subgroups of $S$
containing such elements are of type $N_{2}^{-}$ or ${\rm GU}_{4}(q)$, so we may assume that $H$ is
one of these subgroups. Now let $x_2 \in S$ be a regular semisimple element that fixes an
orthogonal decomposition $2^{+}\perp 6^{+}$ and lifts to an element in $\O_{8}^{+}(q)$ of order
$q^3-1$. Note that $x_2$ is a derangement, and let $a_2$ be the number of $A$-classes of such
elements. Then
$$a_2 \geqs \left\lceil \frac{\phi(q-1)\phi(q^3-1)}{12\cdot 6d\log_pq}\right\rceil$$
and the result follows if $q>7$. Finally, if $q=5$ or $7$ then one can check directly that there
are at least three $A$-classes of such elements.
\end{proof}

\begin{proposition}\label{p:minus}
The conclusion to Theorem \ref{t:main2} holds if $S={\rm P\O}_{2m}^{-}(q)$ and $m \geqs 4$.
\end{proposition}

\begin{proof}
We may assume that $S$ is not one of the groups in the statement of Proposition \ref{p:ort0}. Let
$x_1 \in S$ be an irreducible element that lifts to an element of order $(q^m+1)/d$ in
$\O_{2m}^{-}(q)$. Let $a_1$ be the number of $A$-classes of such elements. Then
\begin{equation}\label{e:a1bd}
a_1 \geqs \left\lceil \frac{\phi(q^{m}+1)}{2m\cdot 2d\log_pq} \right\rceil
\end{equation}
and thus $a_1 \geqs 3$ (and $a_1$ tends to infinity as $|S|$ tends to infinity). By \cite[Main
Theorem]{Ber}, if $H$ is a maximal subgroup of $S$ that contains such an element then $H$ is a
field extension subgroup of type ${\rm O}_{2m/k}^{-}(q^k)$ or ${\rm GU}_{m}(q)$ (with $m$ odd),
where $k$ is a prime divisor of $m$. In both of these cases, any element $x_2 \in S$ that fixes a
decomposition $2^{+} \perp (2m-2)^{-}$ of the natural $S$-module, centralising the $2$-space and
acting irreducibly on the $(2m-2)$-space, is a derangement. Now, if $a_2$ denotes the number of
$A$-classes of such elements then
\begin{equation}\label{e:a2bd}
a_2 \geqs \left\lceil \frac{\phi(q^{m-1}+1)}{(2m-2)\cdot 2d\log_pq} \right\rceil
\end{equation}
and the result follows.
\end{proof}

\begin{proposition}\label{p:odd}
The conclusion to Theorem \ref{t:main2} holds if $S=\O_{2m+1}(q)$ and $m \geqs 3$.
\end{proposition}

\begin{proof}
If $S=\O_{7}(3),\O_7(5)$ or $\O_9(3)$ then the result can be checked directly, using {\sc Magma}
\cite{magma}, so we will assume that we are not in one of these cases.

Let $x_1 \in S$ be a regular semisimple element of order $(q^m+1)/2$ that fixes a decomposition
$(2m)^{-}\perp 1$ of the natural $S$-module, and let $a_1$ be the number of distinct $A$-classes of
such elements. Then \eqref{e:a1bd} holds (setting $d=1$), so $a_1 \geqs 3$ (and $a_1$ tends to
infinity as $|S|$ tends to infinity). By \cite[Proposition 5.20]{BGK}, the only maximal subgroup of
$S$ containing such an element is the stabiliser of a nondegenerate $2m$-space of minus-type,
denoted by $H=N_{2m}^{-}$. In this situation, let $x_2 \in S$ be an element of order
$q(q^{m-1}+1)/2$ that fixes a decomposition $3 \perp (2m-2)^{-}$, where $x_2$ acts indecomposably
on the $3$-space and irreducibly on the $(2m-2)^{-}$-space. If $a_2$ denotes the number of
$A$-classes of such elements then \eqref{e:a2bd} holds (with $d=1$), so $a_2 \geqs 2$ and the
result follows since every regular unipotent element is also a derangement.
\end{proof}

\vs

This completes the proof of Theorem \ref{t:main2}.

\section{Two classes of derangements}\label{s:2classes}
In this section we investigate the finite primitive permutation groups $G$ with the property
$\kappa(G)=2$, with the aim of proving Theorem \ref{t:main3}. We begin with a preliminary lemma. As before, if $X$ is a group then $X^*=X
\setminus \{1\}$ is the set of nontrivial elements in $X$.

\begin{lemma}\label{l:frob}
Let $G \leqs {\rm Sym}(\Omega)$ be a finite transitive permutation group with point stabiliser $H
\neq 1$. Let $N$ be a regular normal subgroup of $G$. Then $G$ is a Frobenius group with kernel $N$
if and only if $\Delta(G) \subseteq N$.
\end{lemma}

\begin{proof}
Since $N$ is regular, we have $G=HN$ and $H \cap N = 1$. By definition, if $G$ is a Frobenius group
with kernel $N$ then $\Delta(G) = N^*$.

Now assume $\Delta(G) \subseteq N$. First observe that if $x \in N^*$ then $x^G\cap H\subseteq
N^*\cap H=\emptyset$, so $x \in \Delta(G)$ and thus $N^*\subseteq \Delta(G)$. Therefore $\Delta(G)
= N^*$. Let $\{H_1, \ldots, H_k\}$ be the set of conjugates of $H$ in $G$. Then
$k=|G:\Normalizer_G(H)|\leqs |G:H|=|N|$ and
\[|\bigcup_{i=1}^k H_i^*|\leqs \sum_{i=1}^k|H_i^*|=\sum_{i=1}^k(|H|-1)=k(|H|-1).\]
Now
\[G=\{1\}\cup \Delta(G) \cup \left(\bigcup_{g\in G} (H^*)^g\right)=N\cup \left(\bigcup_{i=1}^k H_i^*\right)\]
and thus
$$|G|= |N|+|\bigcup_{i=1}^k H_i^*| \leqs |N|+k(|H|-1)
 \leqs  |N|+|N|(|H|-1) = |N|\cdot |H|=|G|.$$
Since $|H| \neq 1$, it follows that $k=|N|=|G:H|$ and $|\bigcup_{i=1}^k
H_i^*|=\sum_{i=1}^k|H_i^*|.$ The latter equality forces $H_i\cap H_j=1$ for every $1\leqs i\neq
j\leqs k$. Equivalently, $H\cap H^g=1$ for all $g\in G\setminus H$ and thus $G$ is a Frobenius
group with kernel $N$.
\end{proof}

Recall that if $J$ is a proper subgroup of $G$, then we set
$$\Delta_J(G)=G\setminus \bigcup_{g\in G} J^g.$$
We record the following easy result.

\begin{lemma}\label{l:normal}
Let $H$ be a maximal subgroup of a finite group $G$, $M$ a normal subgroup of $G$ such that $G=HM$, and let $K$ be a proper subgroup of $M$ containing $H\cap M$.
Then $\Delta_K(M)\subseteq \Delta_H(G)$.
\end{lemma}

\begin{proof}
Let $x\in \Delta_K(M)$ and assume that $x\not\in\Delta_H(G)$. Then $x\in H^g$ for some $g\in G$. It follows that $x^{g^{-1}}\in H$ and since $x\in M\normeq G,$ we also have $x^{g^{-1}}\in M,$ so $x^{g^{-1}}\in H\cap M$ and thus $x\in (H\cap M)^g=H^g\cap M$. Since $g\in G=HM$, we can write $g=hm$ with $h\in H$ and $m\in M$. Then $x\in H^g\cap M=H^m\cap M=(H\cap M)^m\leqs K^m$ with
$m\in M$, contradicting our assumption that $x\in \Delta_K(M)$. The result follows.
\end{proof}



\begin{proposition}\label{p:red2}
Let $G \leqs {\rm Sym}(\Omega)$ be a finite primitive permutation group of degree $n$ with point stabiliser $H$. Assume $G$ is not almost simple. If $\kappa(G)=2$, then one of the following holds:
\begin{itemize}\addtolength{\itemsep}{0.2\baselineskip}
\item[{\rm (i)}] $(G,n) = (\Z_3,3)$;
\item[{\rm (ii)}] $G=HN$ is a Frobenius affine group, where the kernel $N$ is an elementary abelian $p$-group of order $n=p^k$ for some odd prime $p$, and $|H|=(n-1)/2$;
\item[{\rm (iii)}] $G$ is a non-Frobenius $2$-transitive affine group.
\end{itemize}
Moreover, any primitive group $G$ as in {\rm (i)} or {\rm (ii)} has the property $\kappa(G)=2$.
\end{proposition}


\begin{proof}
Let $H=G_{\a}$ be a point stabiliser. First assume $G$ is one of the groups in (i) or (ii).
Clearly, $\kappa(G)=2$ in case (i). In (ii), $H$ acts semiregularly on $\Omega \setminus \{\a\}$
with exactly two orbits. In particular, $\Delta(G) = N^* =
x^G \cup y^G$ and thus $\kappa(G)=2$.

Now assume $\kappa(G)=2$. We proceed as in the proof of Theorem \ref{t:red}. Let $N$ be a minimal normal subgroup of $G$, so $G=HN$. If $H=1$ then $G$ is regular and clearly $(G,n) =(\Z_3,3)$ is the only possibility. For the remainder, let us assume $H \neq 1$.

Suppose that $H\cap N\neq 1$. By the proof of Theorem \ref{t:red}, we may assume that
$N\cong S^k$, where $S$ is a nonabelian simple group, $k\geqs 2$ and $G\leqs L\wr S_k$ acting with
its product action on $\Omega=\Gamma^k$, where $L\leqs \rm{Sym}(\Gamma)$ is a primitive almost
simple group with socle $S$. Let $u\in S$ be a derangement on $\Gamma$. Then $x=(u,1,\ldots,1)\in
N$ and $y=(u,u,1,\ldots,1)\in N$ are non-conjugate derangements on $\Omega$. If $k \geqs 3$ then
$z^G=(u,u,u,1,\ldots,1)^G$ would be another $G$-class of derangements, so $k=2$ since
$\kappa(G)=2$. If $S$ has two $L$-classes of derangements with representatives $u$ and $v$, then $(u,1),(u,u),(v,1) \in N$ are non-conjugate derangements, which is a contradiction. Therefore $S$ contains a unique $L$-class of derangements on $\Gamma$.

Write $N=S_1\times S_2$, where $S_i\cong
S$, $i=1,2$. Let $K$ be a maximal subgroup of $N$ such that $H\cap N\leqs K$. By
Lemma \ref{l:normal}, every derangement of $N$ on $N/K$ is also a derangement of $N$ on $\Omega$. It is well known that either $K$ is a diagonal subgroup of the form $\{(s,\phi(s)) \mid
s\in S_1\}$ for some isomorphism $\phi:S_1 \to S_2$, or $K$ is a standard maximal subgroup, that is
$K=S_1\times K_2$ or $K_1\times S_2$, where $K_i<S_i$ is maximal (see, for example, \cite[Lemma
1.3]{Thevenaz}). In the diagonal case, every element of the form $(s,1)$ with $1\neq s\in S_1$ is a
derangement of $N$ on $N/K$. Clearly, this case cannot happen. Now assume $K$ is a standard maximal
subgroup. Without loss of generality, we may assume that $K=K_1\times S_2$, where $K_1$ is maximal
in $S_1$. Let $s\in N$ be a derangement on $N/K$ of prime power order, say $p^e$ for some prime $p$
and integer $e\geqs 1$ (such an element exists by the main theorem of \cite{FKS}). Since
$|\pi(S)|\geqs 3$, choose $a,b\in S_2$ of distinct prime orders that are both different from $p$.
Then $(s,1),(s,a)$ and $(s,b)$ are derangements of $N$ on $N/K$ with distinct orders, so $N$ has at least three distinct $N$-classes of derangements on $N/K$ and thus $N$
has at least three distinct $G$-classes of derangements on $\Omega$. We have now eliminated the case $H \cap N \neq 1$.

Finally, suppose that $H \cap N = 1$, so $N$ is regular and we may identify $\O$ with $N$. By arguing as in the proof of Theorem \ref{t:red}, we deduce that $N$ is an elementary abelian $p$-group for some prime $p$, say $|N|=n=p^k$. In particular, $G$ is an affine group. If $\Delta(G) \subseteq N$ then $G$ is Frobenius by Lemma \ref{l:frob}, and we deduce that (ii) holds (here $H$ acts semiregularly on $\Omega \setminus \{\a\}$, with exactly two orbits). On the other hand, if $\Delta(G)\nsubseteq N$ then $N^*=x^G\subset \Delta(G)$ for some $x \in N^*$, and thus $H$ acts transitively on $N^*$, so $G$ is a $2$-transitive affine group.
\end{proof}

To complete the proof of Theorem \ref{t:main3}, we may assume that $G$ is a non-Frobenius $2$-transitive affine group. Write $G=HN$, where $H=G_{\a}$ and $N$ is a regular normal elementary abelian subgroup of order $p^k$ ($p$ prime). Assume that $\kappa(G)=2$, so $N^*=x^G$ and $\Delta(G) = x^G \cup y^G$ for some $x \in N^*$ and $y \in G \setminus N$. Note that $N \leqs \Centralizer_G(x)\leqs G=HN$ and $|x^G|=|G:\Centralizer_G(x)|=|N^*|=p^k-1$, so $\Centralizer_G(x)=N\Centralizer_H(x)$ and
\begin{equation}\label{e:oH}
|H|=|G:N|=|G:\Centralizer_G(x)|\cdot |\Centralizer_G(x):N|=(p^k-1)|\Centralizer_H(x)|.
\end{equation}
We need a couple of preliminary results.

\begin{lemma}\label{l:2ta1}
Let $C = \Centralizer_{H}(x)$. Then $|C| = p^br^c$, where $r \neq p$ is a prime and $b,c \geqs 0$.
\end{lemma}

\begin{proof}
If $\Centralizer_H(x)$ is a $p$-group, then we are done. Assume that $|\Centralizer_H(x)|$ is
divisible by a prime $r\neq p$. Then $\Centralizer_H(x)$ contains an element of $u$ order $r$.
Let $z:=xu\in G$. Then $|z|=pr$ and $z$ is a derangement. Indeed, if $z\in
H^g$ for some $g\in G$, then $z^r=x^r\in H^g$, which implies that $\la x^r\ra=\la x\ra\leqs H^g$ as $(r,p)=1$, so $x\in H^g$ and this is a contradiction since $x \in \Delta(G)$. Since
$\kappa(G)=2$ and $|z|\neq |x|$, we
must have $z^G=y^G$. Therefore $r$ is uniquely determined and the result follows.
\end{proof}

In the next lemma, note that part (i) holds for \emph{any} non-Frobenius $2$-transitive group $G=HN$ such that $|N|=p^k$ and $p$ divides $|H|$.

\begin{lemma}\label{l:2ta2}
Let $H_p$ be a Sylow $p$-subgroup of $H$, and assume that $H_p \neq 1$.
\begin{itemize}\addtolength{\itemsep}{0.2\baselineskip}
\item[{\rm (i)}] $[N,H_p]$ is a proper subgroup of $N$, and $tz \in \Delta(G)$ for all $t \in H_p$ and all $z \in N \setminus [N,H_p]$.
\item[{\rm (ii)}] $H_p$ has exponent $p$ and $H_p^* \subseteq t^H$ for some $t \in H_p^*$. Furthermore, $|\Centralizer_H(x)|=p^b$ and thus $|H|=(p^k-1)p^b$, for some $b \geqs 1$.
\end{itemize}
\end{lemma}

\begin{proof}
Let $P=NH_p$ and observe that $P$ is a Sylow $p$-subgroup of $G$.

First consider (i). Let $c$ be the nilpotency class of $P$, so if we define $\gamma_0(P)=P$ and $\gamma_{i+1}(P)=[\gamma_i(P),P]$ for all $i \geqs 0$, then
$\gamma_c(P)=1$ and $\gamma_{c-1}(P)\neq 1$. Seeking a contradiction, suppose that $N=[N,H_p]$. Then $N\subseteq
[P,P]=\gamma_1(P)$, so
$$N=[N,H_p]\subseteq [\gamma_1(P),P]=\gamma_2(P)$$
and so on. In this way, we deduce that $N\subseteq \gamma_c(P)=1$, which is a contradiction. Hence $[N,H_p]\neq N$ and we fix an element $z \in N\setminus [N,H_p]$.

We claim that $tz \in \Delta(G)$ for all $t \in H_p$. Assume otherwise. Then $tz \in H^g$ for some $g\in G$. Since $G=HN$, we can write $g=hn$ with $h\in H, n\in N$. Then $tz\in
H^n$ and thus $ntzn^{-1}\in H$. Since $z,n\in N$ we have $zn=nz$ and
$$ntzn^{-1}=t (t^{-1}nt n^{-1})z=t[t,n^{-1}]z\in H.$$
Hence $[t,n^{-1}]z\in H\cap N=1$, which implies that $z=[n^{-1},t]\in [N,H_p]$, contradicting our choice of $z$. This completes the proof of part (i). 

Now let us turn to (ii). By (i), $tz \in \Delta(G)$ for all $t \in H_p$. If $t \in H_p^*$ then $tz\not\in N^*=x^G$, so $t z\in y^G$ and thus $(H_p^*)z\subseteq y^G$.

Let $s,t \in H_p^*$. Then $s z, tz \in y^G$, so $(tz)^g=sz$ for some $g=hn\in HN=G$ with $h\in H, n\in N$. It follows that $n^{-1}h^{-1}tzhn=sz$ so
$$s^{-1}t^h=n^s zn^{-1}(z^h)^{-1}\in H\cap N=1,$$
and thus $t^h=s$. Therefore $H_p^* \subseteq t^H$, so all elements in $H_p^*$ have the same order, which must be $p$.

Now $tz \in y^G$ and $tz\in P=NH_p$, so $y$ is a $p$-element and thus every element in
$\Delta(G)$ has $p$-power order. Let $C=\Centralizer_H(x)$. Suppose $|C|$ is divisible by a
prime $r\neq p$ and let $u \in C$ be an element of order $r$. Then $ux \in \Delta(G)$ has order $rp$, which is a contradiction. Therefore $|C|=p^b$ for some $b \geqs 1$, and the result follows (see \eqref{e:oH}).
\end{proof}

Let $G=HN \leqs {\rm Sym}(\O)$ be a primitive affine permutation group, where $|N|=p^k$ for a prime $p$. We may identify $\O$ with $N \cong (\mathbb{F}_p)^k$ and take $H$ to be the stabiliser of the zero vector in $N$, so $H \leqs {\rm GL}_{k}(p)$ is irreducible. The $2$-transitive affine permutation groups were classified by Hering \cite{Hering, Hering2} (also see \cite[Section 7.3]{Cam} and \cite[Appendix 1]{Liebeck}). Four infinite families arise, together with finitely many sporadic cases of degree at most $59^2$. By inspecting these cases, we can severely restrict the possibilities for a non-Frobenius $2$-transitive affine group $G$ with $\kappa(G)=2$.

For the remainder of this section, we will write $\mathcal{P}(n,i)$ for the $i$-th primitive permutation group of degree $n$ in the library of primitive groups stored in {\sc Magma} \cite{magma}, which can be accessed via the command \textsf{PrimitiveGroup}$(n,i)$.

\begin{proposition}\label{p:2ta}
Let $G=HN$ be a non-Frobenius $2$-transitive affine group of degree $p^k$, where $H \leqs {\rm GL}_{k}(p)$ as above. If $\kappa(G)=2$ then one of the following holds:
\begin{itemize}\addtolength{\itemsep}{0.2\baselineskip}
\item[{\rm (i)}] $H \leqs {\rm \Gamma L}_{1}(p^k)$;
\item[{\rm (ii)}] ${\rm SL}_{2}(q) \normeq H$, where $q^2 = p^k$;
\item[{\rm (iii)}] $G = \mathcal{P}(5^2,17) = 5^2{:}(2^{1+2}.6)$, $\mathcal{P}(11^2,42) = 11^2{:}(2^{1+2}.[30])$, $\mathcal{P}(3^4,70) = 3^4{:}((2\times Q_8){:}2){:}5$ or $\mathcal{P}(29^2,104) = 29^2{:}(7 \times 2.{\rm SL}_{2}(5))$.
\end{itemize}
Moreover, each group $G$ in {\rm (iii)} has the property $\kappa(G)=2$.
\end{proposition}

\renewcommand{\arraystretch}{1.2}
\begin{table}
$$\begin{array}{llll} \hline
& n & H & \mbox{Conditions} \\ \hline
{\rm (i)} & p^k & H \leqs {\rm \Gamma L}_{1}(p^k) & \\
{\rm (ii)} & q^a & {\rm SL}_{a}(q) \normeq H \leqs {\rm \Gamma L}_{a}(q) & a \geqs 2 \\
{\rm (iii)} & q^a & {\rm Sp}_{a}(q) \normeq H & a \geqs 4 \\
{\rm (iv)} & q^6 & G_2(q)' \normeq H & p = 2 \\
{\rm (v)} & 5^2, 7^2, 11^2, 23^2 & {\rm SL}_{2}(3) \normeq H & \\
{\rm (vi)} & 3^4 & 2^{1+4} \normeq H & \\
{\rm (vii)} & 9^2, 11^2, 19^2, 29^2, 59^2 & {\rm SL}_{2}(5) \normeq H & \\
{\rm (viii)} & 2^4 & A_6 & \\
{\rm (ix)} & 2^4 & A_7 & \\
{\rm (x)} & 3^6 & {\rm SL}_{2}(13) & \\ \hline
\end{array}$$
\caption{$2$-transitive affine groups}
\label{t:hering}
\end{table}
\renewcommand{\arraystretch}{1}

\begin{proof}
By Hering's Theorem, the possibilities for $H$ are given in Table \ref{t:hering}, where $n=p^k$ denotes the degree of $G$. In order to prove the proposition, we need to eliminate cases (iii) -- (x), and also case (ii) with $a \geqs 3$.

As before, write $N^*=x^G$ and $\Delta(G) = x^G \cup y^G$, where $y \in G \setminus N$. Set $C = \Centralizer_H(x)$ and recall that $|H| = (p^k-1)|C|$. By Lemma \ref{l:2ta1}, it follows that
\begin{equation}\label{e:11}
\mbox{$\frac{|H|_{p'}}{p^k-1}$ is a prime power.}
\end{equation}
We start by considering the cases (ii), (iii) and (iv). Write $q=p^m$, so $ma=k$ and $q^a = p^k$ (where $a=6$ in case (iv)).

Suppose (ii) holds. If $a \geqs 4$ then
$(q^2-1)(q^3-1)$ divides $|H|_{p'}/(p^k-1)$, but this is incompatible with \eqref{e:11}. Now assume $a=3$. Here \eqref{e:11} implies that $q^2-1=r^t$ for some prime $r \neq p$ and integer $t \geqs 0$, so $p^{2m}=1+r^t$ and we deduce that $m=1$ and $p \in \{2,3\}$. If $p=2$ then ${\rm \Gamma L}_3(2)\cong \SL_3(2)$, so $H=\SL_3(2)$, $G=2^3{:}\SL_3(2)$ and using {\sc Magma} we calculate that $\kappa(G) = 5$. Similarly, if
$p=3$ then ${\rm \Gamma L}_3(3)=\GL_3(3)$ and thus $G= 3^3{:}\SL_3(3)$ or $3^3{:}\GL_3(3)$. Here we calculate that $\kappa(G)=10$ or $11$, respectively.

Now assume (iii) holds, so $a\geqs 4$ is even. If $a\geqs 6$ then
$(q^2-1)(q^4-1)$ divides $|H|_{p'}/(p^k-1)$, which contradicts \eqref{e:11}, so we may assume that $a=4$. Here $q^2-1=r^t$, where $r$ is a prime and $t \geqs 0$, so as in the previous case we deduce that $m=1$ and $p \in \{2,3\}$. In particular, $\Sp_4(p)\normeq H \leqs
\GL_4(p)$ with $p=2,3$. If $p=2$ then $H=\Sp_4(2)$ since $\Sp_4(2)$ is a maximal subgroup
of $\GL_4(2)$, so $G = 2^4{:}\Sp_4(2)$ and we calculate that $\kappa(G) = 10$. If $p=3$ then $H\cong \Sp_4(3)$ or $\Normalizer_{\GL_4(3)}(\Sp_4(3)) = \Sp_4(3).2$, and we find that $\kappa(G) =24$ or $18$, respectively.

Next consider (iv). Here $p=2$, $a=6$ and $q^2-1$ divides $|H|_{p'}/(p^k-1)$, so $q^2-1=r^t$ for some prime $r \neq p$ and integer $t \geqs 0$. The only possibility is $m=1$, so $G = 2^{6}{:}G_2(2)'$ or $2^{6}{:}G_2(2)$, and we calculate that $\kappa(G) = 10$ or $14$, respectively.

To complete the proof of the proposition, we need to deal with the remaining cases labelled (v) to (x) in Table \ref{t:hering}. In each of these cases we use the library of primitive groups in {\sc Magma} to determine the possiblities for $G$, and in each case we compute $\kappa(G)$.

Consider (v). Here $k=2$ and ${\rm SL}_{2}(3) \normeq H \leqs {\rm GL}_{2}(p)$, where $p \in \{5,7,11,23\}$. We use the library of primitive groups of degree $p^2$ to determine the possibilities for $G$ with $\kappa(G)=2$; we find that either $p=5$ and $G = \mathcal{P}(5^2,17) = 5^2{:}(2^{1+2}.6)$, or $p=11$ and $G = \mathcal{P}(11^2,42) = 11^2{:}(2^{1+2}.[30])$.
Similarly, in (vi) we find that the only example is $G = \mathcal{P}(3^4,70) = 3^4{:}((2\times Q_8){:}2){:}5$, and in (vii) the only example is $G = \mathcal{P}(29^2,104) = 29^2{:}(7 \times 2.{\rm SL}_{2}(5))$. Finally, in cases (viii), (ix) and (x) we calculate that $\kappa(G) = 5,6$ and $3$, respectively.
\end{proof}

We now focus on the possibilities that can arise in cases (i) and (ii) of Proposition \ref{p:2ta}. We begin with a preliminary lemma.

\begin{lemma}\label{l:sl2}
Let $G$ be the primitive affine group $q^2{:}{\rm SL}_{2}(q)$, where $q=2^m$ and $m \geqs 2$. Then $\kappa(G) \geqs 3$. 
\end{lemma}

\begin{proof}
Write $G=HN$, where $H = {\rm SL}_{2}(q)$ and $N$ is elementary abelian of order $q^2=2^{2m}$. We can embed $G$ into $\SL_3(q)$ as follows: 
\begin{align*}
G & =\left\{\left(\begin{array}{ccc}1 & 0 & 0 \\ \alpha & a & b \\\beta & c & d\end{array}\right) \mid\alpha,\beta,a,b,c,d\in \mathbb{F}_q,\, ad-bc=1\right\} \\
N & =\left\{\left(\begin{array}{ccc}1 & 0 & 0 \\ \alpha & 1 & 0 \\\beta & 0 & 1\end{array}\right) \mid\alpha,\beta\in \mathbb{F}_q \right\}\cong q^2 \\
H & =\left\{\left(\begin{array}{ccc}1 & 0 & 0 \\ 0 & a & b \\ 0 & c & d\end{array}\right) \mid a,b,c,d\in \mathbb{F}_q,\, ad-bc=1\right\}\cong \SL_2(q).
\end{align*}
Note that
$$H_2=\left\{\left(\begin{array}{ccc}1 & 0 & 0 \\ 0 & 1 & 0 \\ 0 & c & 1\end{array}\right) \mid c\in \mathbb{F}_q\right\}$$ 
is a Sylow $2$-subgroup of $H$. Direct computation shows that 
$$[N,H_2]=\left\{\left(\begin{array}{ccc}1 & 0 & 0 \\ 0 & 1 & 0 \\ \beta & 0 & 1\end{array}\right) \mid\beta\in \mathbb{F}_q\right\}.$$

By Lemma \ref{l:2ta2}(i), we deduce that 
$$z_1=\left(\begin{array}{ccc}1 & 0 & 0 \\ 0 & 1 & 0 \\ 0 & 1 & 1\end{array}\right)\left(\begin{array}{ccc}1 & 0 & 0 \\ 1 & 1 & 0 \\ 1 & 0 & 1\end{array}\right) = \left(\begin{array}{ccc}1 & 0 & 0 \\ 1 & 1 & 0 \\ 0 & 1 & 1\end{array}\right)$$ and  
$$z_2=\left(\begin{array}{ccc}1 & 0 & 0 \\ 0 & 1 & 0 \\ 0 & 1 & 1\end{array}\right)\left(\begin{array}{ccc}1 & 0 & 0 \\ \gamma & 1 & 0 \\ \gamma & 0 & 1\end{array}\right) = \left(\begin{array}{ccc}1 & 0 & 0 \\ \gamma & 1 & 0 \\ 0 & 1 & 1\end{array}\right)$$ 
are derangements, where $\gamma\in \mathbb{F}_q$ is a generator for $\mathbb{F}_q^*$. Since $z_1,z_2\not\in N$, it suffices to show that $z_1$ and $z_2$ are not $G$-conjugate. 

Seeking a contradiction, assume that $z_1^g=z_2$ for some $g\in G$, say  
$$g=\left(\begin{array}{ccc}1 & 0 & 0 \\ \alpha & a & b \\ \beta & c & d\end{array}\right)$$ 
where $a,b,c,d,\alpha,\beta\in \mathbb{F}_q$ and $ad-bc=1$. Now it follows from the equation $z_1^g=z_2$ that $z_1g=gz_2$ and hence 
\[\left(\begin{array}{ccc}1 & 0 & 0 \\\alpha+1 & a & b \\\alpha+\beta & a+c & b+d\end{array}\right)=\left(\begin{array}{ccc}1 & 0 & 0 \\\alpha+\gamma a & a +b& b \\\beta+\gamma c & c+d & d\end{array}\right)\]
which implies that \[\left\{ \begin{array}{ccc} \alpha+1&=&\alpha+\gamma a \\
a&=&a+b\\
\alpha+\beta&=&\beta+\gamma c\\
a+c&=&c+d\\
b+d&=&d\end{array}\right.\]
and thus
\[\left\{ \begin{array}{ccc} \gamma a&=&1 \\
b&=&0\\
d&=&a\\
\alpha&=&\gamma c.\end{array}\right.\]
Since $ad-bc=1$ we deduce that  
$$1=ad=a^2=\gamma^{-2}$$ 
and thus $\gamma^2=1$. This implies that $\gamma=1$, which is a contradiction since $m \geqs 2$.
\end{proof}

\begin{proposition}\label{p:2ta3}
Let $G=HN$ be a non-Frobenius $2$-transitive affine group of degree $p^k$, where ${\rm SL}_{2}(q) \normeq H$ and $q^2=p^k$. Then $\kappa(G)=2$ if and only if $G \cong S_4$.
\end{proposition}

\begin{proof}
Let us assume that $\kappa(G)=2$ and write $\Delta(G) = x^G \cup y^G$ as before, where $N^* = x^G$. Set $C = \Centralizer_{H}(x)$ and let $H_p$ be a Sylow $p$-subgroup of $H$. Recall that $|H|=(p^k-1)|C|$ (see \eqref{e:oH}). Write $k=2m$, where $m \geqs 1$ is an integer.

Here $|\SL_2(q)|=p^m(p^k-1)$ divides $|H|$ and thus $p^m$ divides $|C|$, so Lemma \ref{l:2ta2}(ii) implies that $|C|=p^b$ where $b\geqs m$. Therefore, $|H:\SL_2(p^m)|=p^{b-m}$.
Since $|{\rm \Gamma L}_2(p^m):\SL_2(p^m)|=m(p^m-1)$, it follows that $H\leqs \Gamma {\rm SL}_2(p^m)$ and thus $H/\SL_2(p^m)$ is cyclic. More precisely, either $H=\SL_2(p^m)$ or
$H=\SL_2(p^m)\la \tau\ra$ where $\tau$ is a $p$-element and $\tau^{p^{b-m}}\in \SL_2(p^m)$. By Lemma \ref{l:2ta2}(ii), $H_p$ has exponent $p$ and thus $|\tau|=p$.

First assume that $H=\SL_2(p^m)\la \tau\ra$ with $|\tau|=p$. Let $1 \neq \sigma\in H_p\cap \SL_2(p^m)$. By Lemma \ref{l:2ta2}(ii), $\tau=\sigma^h\in \SL_2(p^m)$ for some $h\in H$, which is a contradiction. Therefore, this case does not occur.

Finally, suppose that $H=\SL_2(p^m)$. Here $H_p$ is an elementary abelian $p$-group of order $p^m$. By Lemma \ref{l:2ta2}(ii), all nontrivial elements in $H_p$ are $H$-conjugate and we quickly deduce that $p=2$. If $m=1$ then $G\cong 2^2{:}S_3\cong S_4$ and $\kappa(G)=2$, and Lemma \ref{l:sl2} implies that $\kappa(G) \geqs 3$ if $m \geqs 2$.
\end{proof}

The next proposition completes the proof of Theorem \ref{t:main3}.

\begin{proposition}\label{p:2ta2}
Let $G=HN$ be a non-Frobenius $2$-transitive affine group of degree $p^k$, where $H \leqs {\rm \Gamma L}_{1}(p^k)$. Then $\kappa(G)=2$ only if $k$ is even and $|H|=2(p^k-1)$.
\end{proposition}

\begin{proof}
Suppose $\kappa(G)=2$ and write $\Delta(G) = x^G \cup y^G$ as before, where $N^* = x^G$. Set $C = \Centralizer_{H}(x)$ and let $H_p$ be a Sylow $p$-subgroup of $H$. Note that $H$ is soluble and recall that $|H|=(p^k-1)|C|$.

Set $H_0=H\cap \GL_1(p^k)$ and note that $\GL_1(p^k)$ is cyclic of order $p^k-1$. Then
$H/H_0$ is also cyclic and $|H/H_0|$ divides $k$. Moreover, $NH_0$ is a
Frobenius group, so $H_0\cap C=1$ and $C \cong H_0C/H_0$ is cyclic. Write $|H_0|=(p^k-1)/d$
for some integer $d\geqs 1$. Since $|H|$ is divisible by $p^k-1$, it follows that $d$ divides
$|H:H_0|$. Therefore $H/H_0$ has a normal subgroup of order $d$, and the inverse image of this subgroup in $H$, say $L$, is a normal subgroup of $H$ containing $H_0$, and $|L|=p^k-1$. There are two cases to consider.

First assume that $|C|$ is divisible by $p$. Then $H_p \neq 1$, so Lemma \ref{l:2ta2}(ii) implies
that $H_p$ has exponent $p$ and $C$ is a $p$-group, say $|C|=p^b$. Since $C$ is cyclic we have
$p^b=p$ and thus $|H|=p(p^k-1)$, which implies that $H=LC$, $L\cap C=1$ and $C$ is a Sylow
$p$-subgroup of $H$. Write $C=\la t\ra$ with $|t|=p$. By Lemma \ref{l:2ta2}(ii), $t^h=t^{-1}$ for
some $h \in H$, say $h=t^sl$ where $l\in L$ and $s\in \Z$. Then $t^h=t^l=t^{-1}$ which implies that
$t^{-2}=[t,l]\in L\cap \la t\ra=1$. Therefore $|t|=2=p$ and $|H|=2(2^k-1)$. In particular, $k$ is even.

Now assume that $|C|$ is indivisible by $p$. Then $|C|=r^c$ for some prime $r\neq p$ and $c\geqs 1$ (see Lemma \ref{l:2ta1}). Write $C = \la t \ra$. As in the proof of Lemma \ref{l:2ta1}, $s x\in \Delta(G)\setminus x^G=y^G$ for all $1\neq s \in C$. It follows that $|C|=r$, so $|H|=r(p^k-1)$ and our aim is to show that $r=2$. Since $\{tx,t^{-1}x\}\subseteq y^G$, we deduce that $t^h=t^{-1}$ for some $h\in H$. There are now two cases to consider.

If $t\not\in L$, then $H=LC$ with $L\cap C=1$, and by arguing as above we deduce that $|t|=|C|=r=2$, $k$ is even and $G=(NL).2$, where $NL$ is a $2$-transitive Frobenius group.

Now assume that $t\in L$ for every subgroup $L$ of index
$r$ in $H$ with $H_0\leqs L$. Since $t$ fixes $x\in N^*$ it follows that $t\not\in H_0$, so $tH_0$ is a nontrivial real element of order $r$ in the cyclic group $H/H_0$, which implies that $tH_0$ is an involution and thus $r=2$. Since $t\in L$ and $|L|=p^k-1$ is even, it follows that $p$ is odd. Moreover, since
$H_0\normeq H_0C\normeq L\normeq H$ with $|H:L|=2$, $|H_0C:H_0|=2$ and $|H:H_0|$ dividing $k$, we deduce that $k$ is divisible by $4$. 
\end{proof}

In \cite[Section 15]{Foulser}, Foulser gives detailed information on the precise structure of the $2$-transitive affine groups $G=HN$ with $H \leqs {\rm \Gamma L}_{1}(p^k)$. To close this section, we show that $\kappa(G)=2$ in the special case $H={\rm GL}_{1}(p^k).2$ (with $k$ even). We thank Bob Guralnick for helpful comments on the proof.

\begin{proposition}\label{p:new}
Let $G=HN$ be a non-Frobenius $2$-transitive affine group of degree $p^k$, where $k$ is even and $H = {\rm GL}_{1}(p^k).2 \leqs {\rm \Gamma L}_{1}(p^k)$. Then $\kappa(G)=2$.
\end{proposition}

\begin{proof}
Write $q^2=p^k$ and set $L={\rm GL}_{1}(q^2)$ and $H=L\la \phi \ra$, where $\phi$ is a field automorphism of order $2$. By Theorem \ref{t:main1}, $\kappa(G) \geqs 2$. Moreover, since  $NL = {\rm AGL}_{1}(q^2)$ is sharply $2$-transitive, it suffices to show that there is a unique $G$-class of derangements in the coset $NL\phi$. Let $y \in NL\phi$ be a derangement. To prove the proposition, we will show that $y^G$ meets $N\phi$, and then we prove that any two derangements in $N\phi$ are $G$-conjugate.

Consider $y^2 \in NL$. If $y^2 \in NL \setminus N$ then $y^2$ has a unique fixed point (since $NL$ is Frobenius), which contradicts the fact that $y$ is a derangement. Therefore, $y^2 \in N$. Now $y \in N\ell \phi$ for some $\ell \in L$, so $(\ell\phi)^2 \in N \cap L=1$ and thus $\ell \phi$ is an involution. We claim that $\ell \phi$ is $L$-conjugate to $\phi$. To see this, note that there are precisely $q+1$ involutions in the coset $L\phi$ (involutions correspond to elements in $L$ that are inverted under the action of $\phi$), and we calculate that $|\phi^L| = q+1$. This justifies the claim, and we deduce that $y^g \in N\phi$ for some $g \in G$. Set $z=y^g$.

It is easy to check that there are precisely $q-1$ involutions in the coset $N\phi$, each having $q$ fixed points. Therefore, $z$ is one of $q(q-1)$ elements in $N\phi$ of order at least $3$, and to complete the proof it suffices to show that any two of these elements are $G$-conjugate. Let 
$C=C_{NL}(\phi)$ and note that $z^{NC} \subset N\phi$, so it suffices to show that $|z^{NC}|=q(q-1)$. Now $|C| = |C_N(\phi)||C_L(\phi)| = q(q-1)$ and $|NC| = |N||C|/|N \cap C| = q^3(q-1)/q = q^2(q-1)$, so we need $|C_{NC}(z)|=q$.  Since $z^2 \in N^*$ and $NL$ is a Frobenius group, we have $C_{NL}(z^2) \leqs N$, so $C_{NL}(z^2) = C_N(z^2)$ and thus $C_{NL}(z) = C_{N}(z) = C_{NC}(z)$. Since $z$ acts on $N$ as a field automorphism of order $2$, we deduce that $|C_{N}(z)| = q$ and the result follows.
\end{proof}

\section{Zeros of characters}\label{s:zeros}

Let $G$ be a finite group, let $H$ be a proper subgroup of $G$ and let $H_G = \bigcap_{g \in G}H^g$
denote the core of $H$ in $G$. Set $$\Delta_{H}(G) = G \setminus \bigcup_{g\in G}H^g$$ and let
$\kappa_H(G)$ be the number of conjugacy classes in $\Delta_{H}(G)$. Note that if $H_G=1$ then $G$
is a permutation group on $G/H$, $\Delta_{H}(G)$ is the set of derangements in $G$, and
$\kappa_H(G)=\kappa(G)$ as before. The aim of this section is to prove Theorem \ref{t:main_app}.

Following \cite{Flavell}, a triple $(G,H,L)$ with $L\normeq H\leqs G$ is called a \emph{$W$-triple}
if $H\cap H^g\leqs L$ for every $g\in G\setminus H$. By a theorem of Wielandt, if $(G,H,L)$ is a $W$-triple then
$$M=G\setminus \bigcup_{g\in G} (H \setminus L)^g$$ is a normal subgroup of $G$ and we have $G=HM$
and $H\cap M=L$ (see \cite[Exercise
1, p.347]{SuzII} for a proof using character theory). The normal subgroup $M$ is called
the \emph{kernel} of the $W$-triple $(G,H,L)$. This is a natural generalisation of Frobenius'
theorem.

Let $\chi$ be a complex character of $G$. Recall that $x \in G$ is a \emph{zero} of $\chi$ if
$\chi(x)=0$. Let $n(\chi)$ be the number of $G$-classes on which $\chi$ vanishes. Note that the
conditions $\kappa_H(G)=1$ and $n(1_H^G)=1$ are equivalent, where $1_H^G$ is the permutation
character of $G$.

In the next lemma, we consider the structure of finite groups $G$ that contain a maximal subgroup $H$ such that $\kappa_H(G)=1$.

\begin{lemma}\label{general}
Let $H$ be a maximal subgroup of a finite group $G$ and assume that $\Delta_H(G)=x^G$ for some
$x\in G$. Let $N=H_G$ and $M=\la x^G \ra$. 
\begin{itemize}\addtolength{\itemsep}{0.2\baselineskip}
\item[{\rm (i)}] If $H \normeq G$, then $G$ is a Frobenius group with an abelian odd-order  kernel $H=G'$ of index two;
\item[{\rm (ii)}] If $H \not\normeq G$, then $N \normeq M \normeq G'$ and either $M=G=G'$, or $M\neq G$ and $(G,H,H\cap M)$ is a $W$-triple with kernel $M$.
\end{itemize}
\end{lemma}

\begin{proof}
First assume that $H\normeq G$. Then $G/H\cong \Z_p$ for some prime $p$ as $H$ is normal and
maximal in $G$.  Since $\Delta_H(G)=G\setminus H=x^G$, $G/H$ has exactly two conjugacy classes and
thus $|G:H|=p=2$. Hence, $G=H\cup Hx$ and $Hx=G\setminus H=x^G,$ where  $H\cap Hx=\emptyset$. Thus
$$|x^G|=|G:\Centralizer_G(x)|=|H|=\frac{1}{2}|G|.$$
Therefore, $|\Centralizer_G(x)|=2$ and so $\Centralizer_G(x)=\la x\ra$ is cyclic of order $2$.
Clearly, $G'\leqs H.$ Now, if $h\in H$ then $hx\in Hx=x^G$, so $hx=x^g$ for some $g\in G$ and thus
$h=x^gx^{-1}\in G'$. Therefore $H\leqs G'$ and thus $H=G'$. As $\Normalizer_G(\la
x\ra)=\Centralizer_G(x)=\la x\ra$, we deduce that $G$ is a Frobenius group with Frobenius
complement $\la x\ra$ of order $2$ and a  Frobenius kernel $G'$ of odd order.  Moreover, since each
element $h\in H=G'$ can be written in the form $h=x^gx^{-1}$ for some $g\in G$, we have
\[h^x=x^{-1}x^gx^{-1}x=x^{-1}x^g=x(x^g)^{-1}=h^{-1}.\]
Therefore $x$ inverts every element of $G'$, so $G'$ is abelian.

Now assume $H$ is not normal in $G$. Then $G'\not\leqs H$ and thus $G=HG'$ and $H\cap G'< G'$, so
$x^G\cap G'$ is nonempty. Let $y\in x^G\cap G'$. Clearly, $\Delta_H(G)=y^G=x^G$, hence $M=\la
x^G\ra=\la y^G\ra \normeq G'$ as $y \in G' \normeq G$. Next, we claim that $N\leqs M$. Let $n\in
N$. If $nx\in\bigcup_{g\in G}H^g$ then $nx\in H^z$ for some $z\in G$, but $n\in N=N^z\leqs H^z$ and
thus $x\in n^{-1}(H^z)=H^z$, which is a contradiction. Therefore, $nx\in \Delta_H(G)=x^G$, which
implies that $nx\in M$ and so $n\in M$ as $x\in M$. We conclude that $N\normeq M\normeq G'$, as
claimed.

If $M=G$, then $M=G=G'$ and we are done. Now assume that $M\neq G$. Let
$k=|G:\Normalizer_G(H)|=|G:H|$ and let $\{H^{g_1}, \ldots, H^{g_k}\}$ be the set of distinct
conjugates of $H$ in $G$. Since $x\in M\setminus H$, we deduce that $G=HM$ and thus
$k=|G:H|=|M:L|$, where $L=H\cap M\normeq H$. Observe that
\[G\setminus M=\left(\bigcup_{i=1}^k H^{g_i}\right) \setminus \bigcup_{i=1}^k (H^{g_i}\cap M)=\left(\bigcup_{i=1}^k
H^{g_i}\right)-\bigcup_{i=1}^k (H\cap M)^{g_i}=\bigcup_{i=1}^k (H\setminus L)^{g_i}.\] It follows that
\[|G|-|M|=|G\setminus M|=|\bigcup_{i=1}^k (H\setminus L)^{g_i}|\leqs k|H\setminus L|=k(|H|-|L|).\] Since $|G|=k|H|$ and
$|M|=k|L|$, we deduce that $|G|-|M|=k(|H|-|L|)$ and thus
\[|\bigcup_{i=1}^k (H\setminus L)^{g_i}|=\sum_{i=1}^k|(H\setminus L)^{g_i}|.\]
Therefore, $(H\setminus L)\cap (H\setminus L)^g=\emptyset$ for all
$g\in G\setminus H$, and thus $(G,H,L)$ is a $W$-triple with kernel $M$.
\end{proof}

Recall that if $G \leqs {\rm Sym}(\Omega)$ is a transitive permutation group of degree $n \geqs 2$
with point stabiliser $H$ then $\Delta(G) \geqs |H|$, with equality if and only if $G$ is sharply
$2$-transitive (see \cite{CC}). The next lemma gives a similar lower bound on $|\Delta_H(G)|$ for
any finite group $G$ and proper subgroup $H$.

\begin{lemma}\label{lower bound}
Let $G$ be a finite group, let $H$ be a proper subgroup of $G$, and set $N=H_G$. Then $|\Delta_H(G)|=|\Delta_{H/N}(G/N)|\cdot |N|\geqs |H|$.
\end{lemma}

\begin{proof}
Let $\Omega$ be the set of right cosets of $H/N$ in $G/N$ and note that $G/N$ is a transitive
permutation group on $\Omega$ with point stabiliser $H/N$. Write
$$\Delta_{H/N}(G/N)=\{Na_1,Na_2,\ldots,Na_k\}$$
and note that $k \geqs |H:N|$ (see \cite{CC}), so to complete the proof it suffices to show that $\Delta_H(G)=\bigcup_{i=1}^k{Na_i}$.

Let $n \in N$ and $i\in \{1,2,\ldots, k\}$. If $na_i\in H^g$ for some $g\in G$ then since $n\in
N\normeq G$ and $N\leqs H$, we have $Na_i=Nna_i\in (H/N)^{Ng}$, which is a contradiction.
Conversely, if $a\in \Delta_H(G)$, then $Na\in \Delta_{H/N}(G/N)$ so $Na=Na_j$ for some $j\in
\{1,2,\ldots, k\}$. We conclude that $\Delta_H(G)=\bigcup_{i=1}^k{Na_i}$ and the result follows.
\end{proof}

\begin{lemma}\label{order}
Let $H$ be a maximal subgroup of a finite group $G$ and assume that $\Delta_H(G)=x^G$ for some $x\in G$. Then $x$ is a $p$-element
 and $\Centralizer_G(x)$ is a $p$-group for some prime $p$.
\end{lemma}

\begin{proof}
Let $N=H_G$.  As in the proof of the previous lemma, first note that $G/N$ is a transitive permutation group on
the set of right cosets of $H/N$ in $G/N$. By \cite[Theorem 1]{FKS}, $Nx\in G/N$ is a derangement
of order $p^b$ for some prime $p$ and integer $b\geqs 1$. Write $|x|=p^am$ with $(p,m)=1$ and
$a\geqs 1$. Then $a\geqs b$ and there exist $u,v\in \Z$ with $1=up^a+vm$. We have that
$x^{p^a}=(x^{p^b})^{p^{a-b}}\in N$ and thus $n^{-1}:=x^{up^a}\in N$. Clearly, $nx=x^{mv}\in
\Delta_H(G)= x^G$, so $x^{mv}$ and $x$ have the same order. It follows that $m=1$ and hence $x$ is
a $p$-element (with $|x|=p^a$).

Finally, seeking a contradiction, suppose that $\Centralizer_G(x)$ is not a $p$-group. Let $r\neq p$ be a prime divisor of
$|\Centralizer_G(x)|$ and fix $y\in\Centralizer_G(x)$ with $|y|=r$. Then $|xy|=p^ar$. Since
$(p^a,r)=1$, we can write $1=up^a+vr$ for some $u,v \in \Z$. Assume that $xy\not\in\Delta_H(G)$.
Then $xy\in H^g$ for some $g\in G$. We have that $x^r=(xy)^r\in H^g$, so $x^{vr}\in H^g$ and thus
$x=x^{up^a}x^{vr}=x^{vr}\in H^g$, which is a contradiction. Therefore $xy\in \Delta_H(G)=x^G$, but
this is not possible since $|xy|=r|x|\neq |x|$. We conclude that $\Centralizer_G(x)$ is a
$p$-group, as required.
\end{proof}

\begin{rem}\label{rem1}
Let $G$ be a finite group and let $\chi$ be a nonlinear irreducible character of $G$ such that $\chi=\phi^G$ and $n(\chi)=1$ for some $\phi\in\Irr(H)$ and proper subgroup $H<G$. Then 
$\Delta_H(G)=x^G$ for some $x\in G$ with $\chi(x)=0$, and thus $\kappa_H(G)=1$. However, the condition  $\kappa_{H}(G)=1$ for some subgroup $H$ of $G$ does
\emph{not} imply that $G$ admits an irreducible character $\chi = \phi^G$ for some $\phi \in \Irr(H)$ with the property $n(\chi)=1$. For example, Theorem \ref{t:main1} implies that $\kappa_H(G)=1$ if 
$(G,H)=(A_5,D_{10})$, but no
irreducible character of $H$ can induce irreducibly to $G$.
\end{rem}

We are now in a position to comlete the proof of Theorem \ref{t:main_app}, on the normal structure of finite groups $G$ with an induced irreducible character $\chi$ such that $n(\chi) = 1$. In order to state the result, let us recall that if $N$ is a proper nontrivial normal subgroup of $G$ then $(G,N)$ is  a \emph{Camina pair} if and only if $|\Centralizer_G(g)|=|\Centralizer_{G/N}(Ng)|$ for all $g\in
G\setminus N$. In addition, $G$ is a \emph{Camina group} if $(G,G')$ is a Camina pair.

\begin{rem}
As noted in Remark \ref{rem:main}(b), Theorem \ref{zeros} below can be viewed as a generalisation of \cite[Theorem 9]{Dixon} and \cite[Theorem 1.1]{Qian}, which give partial structural information in the case where $G$ is soluble. More precisely, \cite[Theorem 9(e)]{Dixon} states that if $G$ is a finite soluble group with an imprimitive irreducible character $\chi$ such that $n(\chi)=1$, then $G$ has a normal subgroup $L$ such that $G/L$ is a $2$-transitive Frobenius group of prime power degree. Similarly, assuming $G$ is soluble, parts (2) and (3) in \cite[Theorem 1.1]{Qian} correspond to parts (i)(a,b) in Theorem \ref{zeros} (note that the conclusion in part (1) of \cite[Theorem 1.1]{Qian} coincides with part (i) in Lemma \ref{general}).
\end{rem}

\begin{theorem}\label{zeros}
Let $H$ be a maximal subgroup of a finite group $G$ and assume that $\Delta_H(G)=x^G$ for some $x\in G$. Let $N=H_G$, $M=\la x^G \ra$ and assume that $H$ is not normal in $G$. Then one of the following holds:
\begin{itemize}\addtolength{\itemsep}{0.2\baselineskip}
\item[{\rm (i)}] $G/N$ is a $2$-transitive Frobenius group with an elementary abelian kernel $M/N$ of
order $p^n$ for some prime $p$, and a complement $H/N$ of order $p^n-1$.
Moreover, $x^G=M \setminus N$, $|\Centralizer_G(x)|=p^n$,  $|x^G|=|H|$, $M'=N$ and one of the following holds:

\vspace{1mm}

\begin{itemize}\addtolength{\itemsep}{0.2\baselineskip}
\item[{\rm (a)}] $M$ is a Frobenius group with kernel $M'$ and $p^n=p>2$.
\item[{\rm (b)}] $M$ is a Frobenius group with kernel $K\normeq G$ such that $G/K\cong \SL_2(3)$ and $M/K\cong Q_8$.
\item[{\rm (c)}] $M$ is a Camina $p$-group.
\end{itemize}

\item[{\rm (ii)}] $G/N\cong {\rm L}_2(8){:}3$, $H/N\cong D_{18}{:}3$, $N$ is a nilpotent $7'$-group and $\Centralizer_G(x)=\la x\ra \cong \Z_7$.

\item[{\rm (iii)}] $G/N\cong A_5$, $H/N\cong D_{10}$, $N$ is a $2$-group and $\Centralizer_G(x)=\la x\ra \cong \Z_3$.
\end{itemize}
In particular, if $G=G'$ then either case {\rm (i)(c)} holds with $p^n=11^2$ and $G/N\cong
11^2{:}\SL_2(5)$, or case {\rm (iii)} holds.
\end{theorem}

\begin{proof}
As previously noted, $G/N$ is a primitive permutation group on the set $\Omega$ of right cosets of $H/N$ in
$G/N$, with point stabiliser $H/N$. Clearly, $G/N$ has only one class of derangements on $\Omega$.
By Theorem \ref{t:main1}, we deduce that one of the following holds:
\begin{itemize}\addtolength{\itemsep}{0.2\baselineskip}
\item $G/N$ is a Frobenius group with an elementary
abelian kernel $M/N$ of order $p^n$ for some prime $p$, and a complement $H/N$ of order $p^n-1$;
\item $G/N\cong {\rm L}_2(8){:}3$ and $H/N\cong D_{18}{:}3$;
\item $G/N\cong A_5$ and $H/N\cong D_{10}$.
\end{itemize}
Moreover, $x^r\in N$ with $r=p,7$ or $3$, respectively.

By Lemma \ref{lower bound}, $|\Delta_H(G)|=|x^G|=|G:\Centralizer_G(x)|\geqs |H|$ and thus
$|\Centralizer_G(x)|\leqs |G:H|$.

Suppose that  $G/N\cong {\rm L}_2(8){:}3$. Then $|\Centralizer_G(x)|\leqs |G:H|=28$. By Lemma
\ref{order}, we know that $x$ is a $7$-element and so $\Centralizer_G(x)$ is a $7$-group. Since
$\la x \ra \leqs \Centralizer_G(x)$, we deduce that $\Centralizer_G(x)=\la x \ra$ with $|x|=7$ and
hence $x$ acts fixed point freely on $N$. Thompson's Theorem \cite[Theorem 4.22]{SuzII} now implies
that $N$ is a nilpotent $7'$-group.

Similarly, if $G/N\cong A_5$ then $|G:H|=6$ and $\Centralizer_G(x)=\la x\ra \cong \Z_3$, so $x$
acts fixed point freely on $N$. In this case, $N$ must be a $2$-group by the main theorem of \cite{FT}.

Finally, let us assume that $G/N$ is a Frobenius group with elementary abelian kernel $M/N$ of
order $p^n$. Then $G=HM$ with $H\cap M=N$ and $|H/N|=|M/N|-1=p^n-1$. In terms of permutation
characters, it follows that $(1_H^G)_M=1_N^M$. As $N\normeq M$, $1_N^M$ vanishes on $M\setminus N$
and thus $1_H^G$ also vanishes on this set, which implies that $M\setminus N\subseteq x^G$.
Furthermore, since $x^G\subseteq M$ and $x^G\cap N\subseteq x^G\cap H=\emptyset$, we have
$x^G\subseteq M\setminus N$ and thus $x^G=M\setminus N$. Now
\[|x^G|=|G:\Centralizer_G(x)|=|M\setminus N|=|M|-|N|=|N|(|M:N|-1)=|N|\cdot |H:N|=|H|\]
and $$|\Centralizer_G(x)|=|G:H|=|M:N|=p^n.$$

Let $y\in M\setminus N=x^G$. Then $y=x^g$ for some $g\in G$ and so
$|\Centralizer_G(y)|=|M:N|=|\Centralizer_{G/N}(Ny)|$ (the last equality holds as $G/N$ is a
Frobenius group with an elementary abelian kernel $M/N$ of order $p^n$). Hence
\[|M:N|\leqs |M:M'|\leqs |\Centralizer_M(y)|\leqs |\Centralizer_G(y)|=|M:N|.\]
So
\[|\Centralizer_{M/M'}(yM')|=|M:M'|=|M:N|=|\Centralizer_G(y)|=|\Centralizer_M(y)|.\]
Thus $M$ is a Camina group. Using the classification of Camina groups (see, for example, \cite{Lewis}), one of the following cases holds:

\begin{itemize}\addtolength{\itemsep}{0.2\baselineskip}
\item[{\rm (a)}] $M$ is a Frobenius group with kernel $M'=N$. Since $M/M'$ is elementary abelian,
we deduce that $M/M'$ is cyclic and thus $p^n=p$, which implies that $|H/N|=p-1$ and so $G/N$ is a Frobenius group of order $p(p-1)$.

\item[{\rm (b)}] $M$ is a Frobenius group with complement $Q_8$, and $p=2$. In this case,
$|M:N|=|M:M'|=4$ so $M/N\cong \Z_2\times \Z_2$ and thus $H/N \cong \Z_3$. Let $K$ be the
kernel of $M$. Then $K\normeq G$ and $G/K\cong Q_8{:}3 \cong \SL_2(3)$.

\item[{\rm (c)}] $M$ is a $p$-group.
\end{itemize}

Finally, to complete the proof of the theorem, assume that $G=G'$. Then either case (i)(c) or (iii)
holds. Suppose that case (i)(c) holds. Then $G/N$ is a Frobenius group with a perfect Frobenius
complement $H/N$. By \cite[Theorem A]{M}, we have $H/N\cong\SL_2(5)$. In particular, since
$|H/N|=120=p^n-1$, we deduce that $p^n=121=11^2$.
\end{proof}

\vs

In view of Remark \ref{rem1}, by combining Lemma \ref{general} and Theorem \ref{zeros}, the proof of Theorem \ref{t:main_app} is now complete.





\vs
\end{document}